\documentclass[10pt,a4paper]{article}
\usepackage[ansinew]{inputenc}
\usepackage{amsmath,amssymb,amsthm}
\usepackage{paralist,url}

\def\pg{\mbox{\rm PG}}

\def\gf{\mbox{\rm GF}}
\def\q{\mbox{\rm Q}}
\def\h{\mbox{\rm H}}
\def\w{\mbox{\rm W}}

\newcommand{\erz}[1]{\langle #1 \rangle}

\def\C{\mathcal{C}}
\def\S{\mathcal{S}}
\def\P{\mathcal{P}}
\def\B{\mathcal{B}}
\def\G{\mathcal{G}}
\def\I{\mathrel{\mathrm I}}

\def\P{\mathcal{P}}
\def\L{\mathcal{L}}

\def\M{\mathcal{M}}
\def\Q{\mathcal{Q}}
\def\R{\mathcal{R}}
\def\IS{\S=(\P,\G,\I)}
\def\ISp{\S'=(\P',\G',\I')}

\newtheorem{Th}{Theorem}[section]
\newtheorem{Le}[Th]{Lemma}

\newtheorem{Co}[Th]{Corollary}
\newtheorem{Pro}[Th]{Proposition}

\begin{document}
\title{Sets of generators blocking all generators in finite classical polar
spaces}
\author{Jan De Beule\thanks{This author is a Postdoctoral Fellow of the
Research Foundation Flanders (FWO) (Belgium)}, Anja Hallez, Klaus Metsch and Leo Storme\thanks{This author is
an Alexander von Humboldt Fellow, and wishes to thank the Alexander von Humboldt
Foundation for granting him this fellowship}} 
\date{}

\maketitle

\begin{abstract}
We introduce generator blocking sets of finite classical polar spaces. These
sets are a generalisation of maximal partial spreads. We prove a
characterization of these minimal sets of the polar spaces $\q(2n,q)$,
$\q^-(2n+1,q)$ and $\h(2n,q^2)$, in terms of cones with vertex a subspace
contained in the polar space and with base a generator blocking set in a polar
space of rank 2.
\end{abstract}

{\bf keywords:} partial spreads, blocking sets, finite classical polar spaces.

\noindent {\bf Affiliations:} Jan De Beule, Anja Hallez and Leo Storme: Ghent University, Department of Mathematics, Krijgslaan 281, S22, B-9000, Gent, Belgium. \url{jdebeule@cage.ugent.be}, \url{athallez@cage.ugent.be}, \url{ls@cage.ugent.be}\\
\noindent Klaus Metsch: Universit\"at Gie\ss{}en, Mathematisches Institut, Arndtstra\ss{}e 2, D-35392, Gie\ss{}en, Germany. \url{Klaus.Metsch@math.uni-giessen.de}

\section{Introduction and definitions}
Consider the projective space $\pg(3,q)$. It is well known that a line of
$\pg(3,q)$ is the smallest blocking set with relation to the planes of
$\pg(3,q)$. It is also well known that any blocking set $\B$ with relation to
the planes, such that $|\B| < q+\sqrt{q}+1$, contains a line (\cite{Bruen1970}).

Consider now any symplectic polarity $\varphi$ of $\pg(3,q)$. The points of
$\pg(3,q)$, together with the totally isotropic lines with relation to
$\varphi$, constitute the generalized quadrangle $\w(3,q)$. If $\B$ is a
blocking set with relation to the planes of $\pg(3,q)$, then $\B$ is a set of
points of $\w(3,q)$ such that on any point of $\w(3,q)$ there is at least one
line of $\w(3,q)$ meeting $\B$ in at least one point. Dualizing to the
generalized quadrangle $\q(4,q)$, we find a set $\L$ of lines of $\q(4,q)$ such
that every line of $\q(4,q)$ meets at least one line of $\L$. Together with the
known bounds on blocking sets of $\pg(2,q)$, we observe the following proposition.

\begin{Pro}\label{pro:q4q}
Suppose that $\L$ is a set of lines of $\q(4,q)$ with the property that every
line of $\q(4,q)$ meets at least one line of $\L$. If $|\L|$ is smaller than the
size of a non-trivial blocking set of $\pg(2,q)$, then $\L$ contains the pencil of
$q+1$ lines through a point of $\q(4,q)$ or $\L$ contains a regulus contained in
$\q(4,q)$.
\end{Pro}

This proposition motivates the study of small sets of generators of three
particular finite classical polar spaces, meeting every generator. In this
section, we define generalized quadrangles and describe briefly the finite
classical polar spaces, and we state the main theorems to be proved in the
paper.

A (finite) \emph{generalized quadrangle} (GQ) is an
incidence structure $\IS$ in which $\P$ and $\G$ are disjoint non-empty sets of
objects called {\em points} and {\em lines} (respectively), and for which $\I
\subseteq (\P \times \G) \cup (\G \times \P)$ is a symmetric point-line
incidence relation satisfying the following axioms:
\begin{compactenum}[\rm (i)]
\item Each point is incident with $1+t$ lines $(t \geq 1)$ and
two distinct points are incident with at most one line.
\item Each line is incident with $1+s$ points $(s \geq 1)$ and
two distinct lines are incident with at most one point.
\item If $X$ is a point and $l$ is a line not incident with $x$,   
then there is a unique pair $(Y,m) \in \P \times \G$ for which $X \I m \I 
Y \I l$. 
\end{compactenum}
The integers $s$ and $t$ are the {\em parameters} of the GQ and $\S$ is said to
have {\em order} $(s,t)$. If $\IS$ is a GQ of order $(s,t)$, we say that $\ISp$
is a {\em subquadrangle} of order $(s',t')$ if and only if  $\P' \subseteq \P$,
$\G' \subseteq \G$, and $\ISp$ is a generalized quadrangle with $\I'$ the
restriction of $\I$ to $\P' \times \G'$.

The {\em finite classical polar spaces} are the geometries consisting of the
totally isotropic, respectively, totally singular, subspaces of non-degenerate
sesquilinear, respectively, non-degenerate quadratic forms on a projective space
$\pg(n,q)$. So these geometries are the non-singular symplectic polar spaces
$\w(2n+1,q)$, the non-singular parabolic quadrics $\q(2n,q)$, $n\geq 2$, the
non-singular elliptic and hyperbolic quadrics $\q^-(2n+1,q)$, $n\geq 2$, and
$\q^+(2n+1,q)$, $n\geq 1$, respectively, and the non-singular hermitian
varieties $\h(d,q^2)$, $d\geq 3$. For $q$ even, the parabolic polar space
$\q(2n,q)$ is isomorphic to the symplectic polar space $\w(2n-1,q)$. For our
purposes, it is sufficient to recall that every non-singular parabolic quadric
in $\pg(2n,q)$ can, up to a coordinate transformation be described as the set of
projective points satisfying the equation $X_0^2+X_1X_2+\ldots+X_{2n-1}X_{2n} =
0$. Every non-singular elliptic quadric of $\pg(2n+1,q)$ can up to a coordinate
transformation be described as the set of projective points
satisfying the equation $g(X_0,X_1)+X_2X_3+\ldots+X_{2n}X_{2n+1} = 0$,
$g(X_0,X_1)$ an irreducible homogeneous quadratic polynomial over $\gf(q)$.
Finally, the hermitian variety $\h(n,q^2)$ can up to a coordinate
transformation be described as the set of projective points
satisfying the equation $X_0^{q+1}+X_1^{q+1}+\ldots + X_n^{q+1} = 0$.

The {\em generators} of a classical polar space are the projective subspaces of
maximal dimension completely contained in this polar space. If the generators
are of dimension $r-1$, then the polar space is said to be of {\em rank} $r$.

Finite classical polar spaces of rank $2$ are examples of generalized
quadrangles, and are called {\em finite classical generalized quadrangles}.
These are the non-singular parabolic quadrics $\q(4,q)$, the
non-singular elliptic quadrics $\q^-(5,q)$, the non-singular hyperbolic quadrics
$\q^+(3,q)$, the non-singular hermitian varieties $\h(3,q^2)$ and $\h(4,q^2)$,
and the symplectic generalized quadrangles $\w(3,q)$ in $\pg(3,q)$. The
GQs $\q(4,q)$ and $\w(3,q)$ are dual to each other, and have both order $(q,q)$.
The GQs $\q(4,q)$ and $\w(3,q)$ are self-dual if and only if $q$ is even.
Finally, the GQs $\h(3,q^2)$ and $\q^-(5,q)$ are also dual to each other, and
have respective order $(q^2,q)$ and $(q,q^2)$. The GQ $\h(4,q^2)$ has order
$(q^2,q^3)$, and the GQ $\q^+(3,q)$ has order $(q,1)$. By taking hyperplane
sections in the ambient projective space, it is clear that $\q^+(3,q)$ is a
subquadrangle of $\q(4,q)$, that $\q(4,q)$ is a subquadrangle of $\q^-(5,q)$,
and that $\h(3,q^2)$ is a subquadrangle of $\h(4,q^2)$. These well known facts
can be found in e.g. \cite{PT2009}.

Consider a finite classical polar space $\S$ of rank $r \geq 2$. A set $\L$ of
generators of $\S$ is called a {\em generator blocking set} if it has the
property that every generator of $\S$ meets at least one element of $\L$
non-trivially. We generalize this definition to non-classical GQs, and we say
that $\L$ is a generator blocking set of a GQ $\S$ if $\L$ has the property that
every line of $\S$ meets at least one element of $\L$. Clearly, for finite
classical generalized quadrangles, both definitions coincide. Suppose that $\L$
is a generator blocking set of a finite classical polar space, respectively a
GQ. We call an element $\pi$ of $\L$ {\em essential} if and only if there exists
a generator, respectively line, of $\S$ not in $\L$, meeting no element of $\L
\setminus \{\pi\}$. We call $\L$ {\em minimal} if and only if all of its
elements are essential.

A {\em spread} of a finite classical polar space is a set $\C$ of generators
such that every point is contained in exactly one element of $\C$. Hence the generators
in the set $\C$ are pairwise disjoint. A {\em cover} is a set $\C$ of generators such that
every point is contained in at least one element of $\C$. Hence a spread is a cover
consisting of pairwise disjoint generators. From the definitions,
it follows that spreads and covers are particular examples of generator blocking
sets.

In this paper, we will study small generator blocking sets of the polar spaces
$\q(2n,q)$, $\q^-(2n+1,q)$ and $\h(2n,q^2)$, $n \geq 2$, all of rank $n$. The
following theorems, inspired by Proposition~\ref{pro:q4q}, will be proved in
Section~\ref{sec:rank2}.

\begin{Th}\label{th:rank2}
Let $\L$ be a generator blocking set of a finite generalized quadrangle of order
$(s,t)$, with $|\L| = t+1$. Then $\L$ is the pencil of $t+1$ lines
through a point, or $t \geq s$ and $\L$ is a spread of a subquadrangle of order
$(s,t/s)$.
\end{Th}

\begin{Th}\label{th:rank2_gap}
\begin{compactenum}[\rm (a)]
\item Let $\L$ be a generator blocking set of $\q^-(5,q)$, with $|\L| = q^2 +
\delta + 1$. If $\delta \le \frac{1}{2}(3q-\sqrt{5q^2+2q+1})$, then $\L$
contains the pencil of $q^2+1$ generators through a point or $\L$ contains a cover
of $\q(4,q)$ embedded as a hyperplane section in $\q^-(5,q)$.
\item Let $\L$ be a generator blocking set of $\h(4,q^2)$, with $|\L| = q^3 +
\delta + 1$. If $\delta < q-3$, then $\L$ contains the pencil of $q^3+1$
generators through a point.
\end{compactenum}
\end{Th}

Section~\ref{sec:rankn} is devoted to a generalization of
Proposition~\ref{pro:q4q} and Theorem~\ref{th:rank2_gap} to finite classical
polar spaces of any rank.

\section{Generalized quadrangles}\label{sec:rank2}

In this section, we study minimal generator blocking sets $\L$ of GQs of order
$(s,t)$. After general observations and the proof of Theorem~\ref{th:rank2}, we
devote two subsections to the particular cases $\S=\q^-(5,q)$ and
$\S=\h(4,q^2)$. We remind that for a GQ $\IS$ of order $(s,t)$, $|\P| =
(st+1)(s+1)$ and $|\G| = (st+1)(t+1)$, see e.g. \cite{PT2009}. Suppose that $P$ is a point
of $\S$, then we denote by $P^\perp$ the set of all points of $\S$ collinear
with $P$. By definition, $P \in P^\perp$. For a classical GQ $\S$
with point set $\P$, the set $P^\perp = \pi \cap \P$, with $\pi$ the tangent
hyperplane to $\S$ in the ambient projective space at the point $P$
\cite{Hirschfeld,PT2009}. Therefore, when $P$ is a point of a classical GQ $\S$,
we also use the notation $P^\perp$ for the tangent hyperplane $\pi$. From the
context, it will always be clear whether $P^\perp$ refers to the point set or to
the tangent hyperplane.

We denote by $\M$ the set of points of $\P$ covered by the lines of $\L$. 
Suppose that $\P \neq \M$, and consider a point $P \in \M \setminus \P$.
Since a GQ does not contain triangles, different lines on $P$ meet different 
lines of $\L$. As every point lies on $t+1$ lines, this implies that $|\L| = t+1+\delta$ 
with $\delta \ge 0$. For each point $P \in \M$, we define $w(P)$ as the number of 
lines of $\L$ on $P$. Also, we define
\[
W := \sum_{P \in \M} (w(P)-1),
\]
then clearly $|\M| = |\L|(s+1)-W$.

We denote by $b_i$ the number of lines of $\G \setminus \L$ that meet exactly
$i$ lines of $\L$, $0 \leq i$. Derived from this notation, we denote by $b_i(P)$
the number of lines on $P \not \in \M$ that meet exactly $i$ lines of $\L$,
$1\leq i$. Remark that there is no a priori upper bound on the number of lines
of $\L$ that meet a line of $\G \setminus \L$. In the next lemmas however, we
will search for completely covered lines not in $\L$, and therefore we denote by
$\tilde{b_i}$ the number of lines of $\G \setminus \L$ that contain exactly $i$
covered points, $0 \leq i \leq s+1$, and we denote by $\tilde{b}_i(P)$ the
number of lines on $P \not \in \M$ containing exactly $i$ covered points, $0
\leq i \leq s+1$. 

\begin{Le}\label{Qminus_basic}\label{H_basic}
Suppose that $\delta < s - 1$.
\begin{compactenum}[\rm (a)]
\item Let the point $X \in \P \setminus \M$. Then $\sum_ib_i(X)(i-1)=\delta$ and
\[
\sum_{P \in X^\perp \cap \M}(w(P) - 1) \leq \delta.
\]
\item A line not contained in $\M$ can meet at most $\delta+1$ lines of $\L$. In
particular, $\tilde{b}_i = b_i=0$ for $i=0$ and for $\delta+1<i<s+1$.
\item 
\[
\sum_{i=2}^{\delta+1}\tilde{b}_i(i-1) \leq \sum_{i=2}^{\delta+1}b_i(i-1).
\]
\item If $P_0$ is a point of $\M$ that lies on a line $l$ meeting $\M$ only in
$P_0$, then
\[
\sum_{P \in \M \setminus P_0^\perp}(w(P)-1)\le\delta s.
\]
\item 
\[
(s-\delta)\sum_{i=1}^{\delta+1}b_i(i-1)\le (st-t-\delta)(s+1)\delta+W\delta.
\]
\item If not all lines on a point $P$ belong to $\L$, then at most $\delta+1$
lines on $P$ belong to $\L$, and less than $\frac{t}{s}+1$ lines
on $P$ not in $\L$ are completely contained in $\M$.
\end{compactenum}
\end{Le}
\begin{proof}
\begin{compactenum}[\rm (a)]
\item Consider a point $X \in \P \setminus \M$. Each of the $t+1$ lines on $X$
meets a line of $\L$, and every line of $\L$ meets exactly one of these $t+1$
lines. Hence 
\[
|X^\perp\cap \M| \ge  t+1 =\sum_ib_i(X) \,.
\]
Furthermore,
\[
\sum_{P\in X^\perp\cap \M}w(P) = \sum_ib_i(X)i=|{\L}|=t+1+\delta \,.
\]
Both assertions follow immediately.
\item Since every line of $\S$ meets a line of $\L$, it follows that
$\tilde{b}_0=b_0=0$. Consider any line $l \not \in \L$ containing a
point $P \not \in \M$. The $t$ lines different from $l$ on $P$ are blocked by at
least $t$ lines of $\L$ not meeting $l$. So at most $|\L| - t = \delta + 1$
lines of $\L$ can meet $l$.
\item 
Consider a line $l$ containing $i$ covered points with $0 < i \leq \delta + 1$.
Then $l$ must meet at least $i$ lines of $\L$, and, by (b), at most $\delta +1$ 
lines of $\L$. On the left hand side, this line
is counted exactly $i-1$ times, on the right hand side this line is counted at
least $i-1$ times. This gives the inequality.
\item
Each point $P$, with $P\not \in P_0^\perp$, is collinear to exactly one point
$X\neq P_0$ of $l$. For $X \in l, X \neq P_0$, the inequality of (a) gives
$\sum_{P\in X^\perp \cap \M}(w(P)-1)\leq \delta$. Summing over the $s$ points on
$l$ different from $P_0$ gives the expression.
\item It follows from (b) that every line with a point not in $\M$ has at least
$s-\delta$ points not in $\M$. Taking the sum over all points $P$ not in $\M$
and using the equality of (a), one finds 
\[
\sum_{i=1}^{\delta+1}b_i(s-\delta)(i-1)
\le \sum_{P\not\in\M}\sum_{i=1}^{\delta+1}b_i(P)(i-1) = (|\P|-|\M|)\delta.
\]
As $|\M|=|\L|(s+1)-W$, the assertion follows.
\item Suppose that the point $P$ lies on exactly $x\ge 1$ lines that are not
elements of $\L$. It is not possible that all these $x$ lines are contained in
$\M$, since this would require $xs$ lines of $\L$ that are not on $P$, and then
$|\L|\ge t+1-x+xs\ge t+s$, a contradiction with $\delta < s-1$. Thus we find a
point $P_0\in P^\perp\setminus\M$. Then the $t$ lines on $P_0$, different from
$\langle P, P_0 \rangle$ must be blocked by a line of $\L$ 
not on $P$, hence at most $\delta+1$ lines of $\L$ can contain $P$.

If $y$ lines on $P$ do not belong to $\L$, but are completely contained in $\M$,
then at least $1+ys$ lines contained in $\L$ meet the union of these $y$ lines,
so $1+ys \leq |\L| = t+1+\delta$, so $y < \frac{t}{s}+1$ as $\delta < s$.
\end{compactenum}
\end{proof}

\begin{Le}\label{le:twopencil}
Suppose that $\delta=0$. If two lines of $\L$ meet, then $\L$ is a pencil of
$t+1$ lines through a point $P$.
\end{Le}
\begin{proof}
The lemma follows immediately from Lemma~\ref{Qminus_basic}~(f).
\end{proof}

\begin{Le}\label{le:pencil_sub}
Suppose that $\delta=0$. If $\L$ is not a pencil, then $t \geq s$ and $\L$ is a
spread of a subquadrangle of order $(s,t/s)$.
\end{Le}
\begin{proof}
We may suppose that $\L$ is not a pencil, so that the lines of $\L$ are pairwise
skew by Lemma \ref{le:twopencil}. Consider the set $\G'$ of all lines completely
contained in $\M$. The set $\G'$ contains at least all the elements of $\L$, so
$\G'$ is not empty. If $l\in\G'$ and $P\in\M$ not on $l$, then there is a
unique line $g \in \G$ on $P$ meeting $l$. As this line contains already
two points of $\M$, it is contained in $\M$ by Lemma~\ref{Qminus_basic} (b),
that is $g \in \G'$. This shows that $(\M, \G')$ is a GQ of some order $(s, t')$
and hence it has $(t's + 1)(s + 1)$ points. As $|\M| = (t + 1)(s + 1)$, then $t's
= t$, that is $t' = t/s$ and hence $t \geq s$. 
\end{proof}

This lemma proves Theorem~\ref{th:rank2}. 

\subsection{The case $\S = \q^-(5,q)$}
In this subsection, $\S = \q^-(5,q)$, so $(s,t) = (q,q^2)$, and
$|\L|=q^2+1+\delta$. We suppose that $\L$ contains no pencil and we will show
for small $\delta$ that $\L$ contains a cover of a parabolic quadric $\q(4,q)
\subseteq \S$.

The set $\M$ of covered points blocks all the lines of $\q^-(5,q)$. An easy
counting argument shows that $|\M| \geq q^3+1$ (in fact, it follows from
\cite{M1998} that $|\M| \geq q^3+q$, but we will not use this stronger lower
bound). Thus $W = |\L|(q+1) - |\M| \leq (q+1)(q+\delta)$. 
\begin{Le}\label{Wbound}
If $\delta\le\frac{q-1}{2}$, then $W\le \delta(q+2)$.
\end{Le}
\begin{proof}
Denote by $\B$ the set of all lines not in $\L$, meeting exactly $i$ lines of
$\L$ for some $i$, with $2\le i\le \delta+1$. We count the number of pairs
$(l,m)$, $l \in \L$, $m \in \B$, $l$ meets $m$. The number of these pairs is
$\sum_{i=2}^{\delta+1} b_i i$.

It follows from
Lemma~\ref{Qminus_basic} (e), $W \leq (q+1)(q+\delta)$, and $\delta \leq
\frac{q-1}{2}$, that

\begin{eqnarray*}
\sum_{i=2}^{\delta+1}b_ii
&\le &2\sum_{i=1}^{\delta+1}b_i(i-1)\le
2\cdot \frac{(q^3-q^2-\delta)(q+1)\delta+W\delta}{q-\delta}
\\
&\le& 2\frac{(q+1)\delta(q^3-q^2+q)}{q-\delta}
\le 2(q-1)(q^3-q^2+q)=:c
\end{eqnarray*}

Hence, some line $l$ of $\L$ meets at most $\lfloor c/|\L|\rfloor$ lines
of $\B$. Denote by $\B_1$ the set of lines not in $\L$ that
meet exactly one line of $\L$. If a point $P$ does not lie on a line of $\B_1$,
then it lies on at least $q^2-q-\delta$ lines of $\B$ (by
Lemma~\ref{Qminus_basic} (f) and since $\L$ contains no pencil). As $\delta \le
\frac{q-1}{2}$, then $c/|\L|< 2(q^2-q-\delta)$, so at most one point of $l$ can
have this property. Thus $l$ has $x\ge q$ points $P_0$ that lie on a line of
$\B_1$, so $l$ is the only line of $\L$ meeting such a line. Apply
Lemma~\ref{Qminus_basic}~(d) on these $x$ points. As every point not on $l$ is
collinear with at most one of these $x$ points, it follows that
\[
\sum_{P\in\M\setminus l}(w(P)-1)\le\frac{x\delta q}{x-1}\le \frac{\delta
q^2}{q-1}<\delta(q+1)+1\,.
\]
All but at most one point of $l$ lie on a line of $\B_1$, so $l$ is the only line
of $\L$ on these points. One point of $l$ can be contained in more than one line
of $\L$, but then it is contained in at most $\delta+1$ lines of $\L$ by
Lemma~\ref{Qminus_basic}~(f). Hence $\sum_{P\in 
l}(w(P)-1)\le \delta$, and therefore $W\le\delta(q+2)$.
\end{proof} 

\begin{Le}\label{bound_on_bqplus1}
If $\delta\le\frac{q-1}{2}$, then
\[
\tilde{b}_{q+1}\ge q^3+q-\delta-\frac{(q^3+q^2-q\delta-q+1)\delta}{q-\delta}.
\]
\end{Le}
\begin{proof}
We count the number of incident pairs $(P,l)$, $P \in \M$ and $l$ a line of $\q^-(5,q)$,
to see
\[
|\M|(q^2+1) = |\L|(q+1) + \sum_{i=1}^{q+1}\tilde{b}_i i \,.
\]
As $\q^-(5,q)$ has
$(q^2+1)(q^3+1)=|\L|+\sum_{i=1}^{q+1}\tilde{b}_i$ lines, then 
\begin{eqnarray*}
|\L|q+\sum_{i=1}^{q+1}\tilde{b}_i(i-1)&=&
|\L|(q+1)+\sum_{i=1}^{q+1}\tilde{b}_ii-(q^2+1)(q^3+1)
\\
&=&|\M|(q^2+1)-(q^2+1)(q^3+1)
\\
&=& (q^2+1)(q+1)(q+\delta)-W(q^2+1).
\\
    &\ge & (q^2+1)(q+1)q-\delta(q^2+1),
\end{eqnarray*}
where we used $W \leq \delta(q+2)$ from Lemma~\ref{Wbound}. From Lemmas
\ref{Qminus_basic}~(c) and (e) and $W \leq \delta(q+2)$, we have 
\[
(q-\delta)\sum_{i=2}^{\delta+1}\tilde{b}_i(i-1) \leq (q-\delta)\sum_{i=2}^{\delta+1}b_i(i-1)\le
(q^3-q^2)(q+1)\delta+\delta^2.
\]
Together this gives
\[
(|\L|+\tilde{b}_{q+1})q\ge
(q^2+1)(q+1)q-\delta(q^2+1)-\frac{(q^3-q^2)(q+1)\delta+\delta^2}{q-\delta}.
\]
Using $|\L|=q^2+1+\delta$, the assertion follows.
\end{proof}

\begin{Le}\label{le:analyse}
If $\delta\le \frac12(3q-\sqrt{5q^2+2q+1})$, then
$|\L|(|\L|-1)\delta<\tilde{b}_{q+1}(q+1)q$.
\end{Le}
\begin{proof}
First note that the upper bound on $\delta$ implies that
$\delta\le\frac{1}{2}(q-1)$. Using the lower bound for $\tilde b_{q+1}$ from the
previous lemma we find 
\begin{align*}
&
2(q-\delta)\left(\tilde b_{q+1}(q+1)q-|\L|(|\L|-1)\delta\right)
\\
& \ge 2q^4\cdot g(\delta)
+(q-1-2\delta)(-2\delta^2q^2+\delta^2q+3q^4+3q^3+2q^2+q)\\
& +2\delta^4+2\delta^3+q\delta^2+3q^2\delta^2+q+q^2+3q^3+\frac{5}{2} q^4 \,,
\end{align*}
with
\[
g(\delta):=(q^2-\frac{1}{2}q-\frac{1}{4}-3q\delta+\delta^2).
\]
The smaller zero of $g$ is $\delta_1=\frac12(3q-\sqrt{5q^2+2q+1})$. Hence, if
$\delta\le\delta_1$, then $\delta\le \frac12(q-1)$ and $g(\delta)\ge 0$, and
therefore $|\L|(|\L|-1)\delta<\tilde b_{q+1}(q+1)q$.
\end{proof}

\begin{Le}\label{le:q3q}
If $\delta \le \frac{1}{2}(3q-\sqrt{5q^2+2q+1})$, then there exists a
hyperbolic quadric $\q^+(3,q)$ contained in $\M$. 
\end{Le}
\begin{proof}
Count triples $(l_1,l_2,g)$, where $l_1,l_2$ are skew lines of
$\L$ and $g \not \in \L$ is a line meeting $l_1$ and $l_2$ and being completely
contained in $\M$. Then 
\[
|\L|(|\L|-1)z\ge \tilde{b}_{q+1}(q+1)q\,,
\]
where $z$ is the average number of transversals contained in $\M$ and not contained 
in $\L$, of two skew lines of $\L$. By Lemma~\ref{le:analyse}, we find that $z > \delta$.
Hence, we find two skew lines $l_1,l_2\in\L$ such that $\delta+1$ of
their transversals are contained in $\M$. The lines $l_1$ and $l_2$ generate a
hyperbolic quadric $\q^+(3,q)$ contained in $\q^-(5,q)$, denoted by $\Q^+$. If
some point $P$ of $\Q^+$ is not contained in $\M$, then the line on it meeting
$l_1,l_2$ has at least two points in $\M$ and the second line of $\Q^+$ on it
has at least $\delta+1$ points in $\M$. This is not possible (cf.
Lemma~\ref{Qminus_basic} (a)). Hence $\Q^+$ is contained in $\M$. 
\end{proof} 

\begin{Le}\label{le:q4q}
If $\delta \le \frac{1}{2}(3q-\sqrt{5q^2+2q+1})$, then $\M$ contains a
parabolic quadric $\q(4,q)$. 
\end{Le}
\begin{proof}
Lemma~\ref{le:q3q} shows that $\M$ contains a hyperbolic
quadric $\q^+(3,q)$, which will be denoted by $\Q^+$. We also know that
$|\M|=|\L|(q+1)-W\ge q^3+q^2+q+1-\delta$ by Lemma~\ref{Wbound}. There are $q+1$
hyperplanes through $\Q^+$, necessarily intersecting $\q^-(5,q)$ in parabolic
quadrics $\q(4,q)$. 

Hence there exists a parabolic quadric $Q(4,q)$, denoted by $\Q$, containing
$\Q^+$ such that 
\[
c:=|(\Q\setminus\Q^+)\cap\M|\ge \frac{|\M|-(q+1)^2}{q+1}>q^2-q-1\,.
\]
Hence, $c\ge q^2-q$. From now on we mean in this proof by a hole of $\Q$ a point
of $\Q$ that is not in $\M$. Each of the $q^3-q-c$ holes of $\Q$ can be
perpendicular to at most $\delta$ points of $(\Q\setminus\Q^+)\cap\M$ (cf.
Lemma~\ref{Qminus_basic} (a)). Thus we find a point $P\in
(\Q\setminus\Q^+)\cap\M$ that is perpendicular to at most
\[
\frac{(q^3-q-c)\delta}{c}\le q\delta
\]
holes of $\Q$. The point $P$ lies on $q+1$ lines of $\Q$ and if such a line is
not contained in $\M$, then it contains at least $q-\delta$ holes of $\Q$ (cf.
Lemma~\ref{Qminus_basic} (b)). Thus the number of lines of $\Q$ on $P$ that are
not contained in $\M$ is at most $q\delta/(q-\delta)$. The hypothesis on $\delta$
guarantees that this number is less than $q+1-\delta$. Thus, $P$ lies on at
least $r\ge \delta+1$ lines of the set $\Q$ that are contained in $\M$. These
meet $\Q^+$ in $r$ points of the conic $C:=P^\perp\cap \Q^+$. Denote this set of
$r$ points by $C'$.

Assume that $\Q\setminus P^\perp$ contains a hole $R$. For $X\in C'$, the hole
$R$ has a unique neighbor $Y$ on the line $PX$; if this is not the point $X$,
then the line $RY$ has at least two points in $\M$, namely $Y$ and the point
$RY\cap \Q^+$. So if $|R^\perp \cap C'| = \emptyset$, then there are at least $r
\ge \delta + 1$ lines on the hole $R$ with at least two points in $\M$. 
This contradicts Lemma~\ref{Qminus_basic}~(a). Therefore $|R^\perp\cap
C'|\ge r-\delta\ge 1$.  As every point of $C'$ lies on $q+1$ lines of $\Q$, two
of which are in $\Q^+$ and one other is contained in $\M$, then every point of
$C'$ has at most $(q-2)q$ neighbors in $\Q$ that are holes. Counting pairs
$(X,Y)$ of perpendicular points $X\in C'$ and holes $R\in\Q\setminus P^\perp$,
it follows that $\Q\setminus P^\perp$ contains at most $r(q-2)q/(r-\delta)\le
(\delta+1)q(q-2)$ holes. Since $P^\perp\cap\Q$ contains at most $q\delta$
holes, we see that $\Q$ has at most $q\delta+(\delta+1)q(q-2)$ holes. As
$\delta\le (q-1)/2$, this number is less than $\frac{1}{2}q(q^2-1)$. Hence, $c >
|\Q|-|\Q^+|-\frac12q(q^2-1)=\frac12q(q^2-1)$.
It follows that $P$ is perpendicular to at most
\[
\frac{(q^3-q-c)\delta}{c}<\delta
\]
holes of $\Q$. This implies that all $q+1$ lines of $\Q$ on $P$ are contained in
$\M$. Then every hole of $\Q$ must be connected to at least $q+1-\delta$
and thus all points of the conic $C$. Apart from $P$, there is only one such
point in $\Q$, so $\Q$ has at most one hole. Then Lemma~\ref{Qminus_basic}~(a)
shows that $\Q$ has no hole. 
\end{proof}

\begin{Le}\label{le:q5qfinal}
If $\M$ contains a parabolic quadric $\q(4,q)$, denoted by $\Q$, and $|\L|\le
q^2+q$, then $\L$ contains a cover of $\Q$.
\end{Le}
\begin{proof}
Consider a point $P \in \Q$. As $|P^\perp \cap \Q|=q^2+q+1$, some line of $\L$
must contain two points of $P^\perp \cap \Q$. Then this line is contained in
$\Q$ and contains $P$.
\end{proof}

In this subsection we assumed that $\L$ contains no pencil. The assumption that
$\delta \le \frac{1}{2}(3q-\sqrt{5q^2+2q+1})$ then implies that $\L$ contains a
cover of a $\q(4,q)\subseteq \q^-(5,q)$. Hence, we may conclude the following
theorem.

\begin{Th}\label{pr:q5qresult}
If $\L$ is a generator blocking set of $\q^-(5,q)$, $|\L|=q^2+1+\delta$, $\delta
\le \frac{1}{2}(3q-\sqrt{5q^2+2q+1})$, then $\L$ contains a pencil of $q^2+1$
lines through a point or $\L$ contains a cover of an embedded $\q(4,q) \subset
\q^-(5,q)$.
\end{Th}

\subsection{The case $\S=\h(4,q^2)$}
In this subsection, $\S = \h(4,q^2)$, so $(s,t) = (q^2,q^3)$. We suppose that $\L$
contains no pencil and that $|\L|=q^3+1+\delta$, and we show that this implies
that $\delta \geq q-3$.
The set $\M$ of covered points must block all the lines of $\h(4,q^2)$. 
%An easy counting argument shows that $|\M| \geq q^5+1$ 
It follows from \cite{DBM2005} that $|\M| \geq q^5+q^2$, and hence 
$W = |\L|(q^2+1) - |\M| \leq (q^2+1)(q+\delta)$.

\begin{Le}\label{WboundH}
If $\delta < q-1$, then $W\le \delta(q^2+3)$. %changed to q-2.
\end{Le}
\begin{proof}
Denote by $\B$ the set of all lines not in $\L$, meeting exactly $i$ lines of
$\L$ for some $i$, with $2\le i\le \delta+1$. We count the number of pairs
$(l,m)$, $l \in \L$, $m \in \B$, $l$ meets $m$. The number of these pairs is
$\sum_{i=2}^{\delta +1}b_i i$. 

It follows from
Lemma~\ref{H_basic} (e), $W \leq (q^2+1)(q+\delta)$, and $\delta < q-1$, that
\begin{eqnarray*}
\sum_{i=2}^{\delta+1}b_i i & \le & 2\sum_{i=1}^{\delta+1}b_i(i-1) \le
\frac{2(q^5-q^3-\delta)(q^2+1)\delta+2W\delta}{q^2-\delta} \\
&\le&\frac{2(q^2+1)\delta(q^5-q^3+q)}{q^2-\delta}
\le 2(q^6+1)=:c
\end{eqnarray*}

Hence, some line $l$ of $\L$ meets at most $\lfloor c/|\L|\rfloor$
lines of $\B$. Denote by $\B_1$ the set of lines not in $\L$ that 
meet exactly one line of $\L$. If a point $P$ does not lie on a line of $\B_1$,
then it lies on at least $q^3-q-\delta$ lines of $\B$ (by
Lemma~\ref{Qminus_basic} (f) and since $\L$ contains no pencil). As $\delta <
q-1$, then $c/|\L|< 3(q^3-q-\delta)$, so at most two points of $l$
can have this property. Thus $l$ has $x\ge q^2-1$ points $P_0$ that lie on a
line of $B_1$, so $l$ is the only line of $\L$ meeting such a line. Apply
Lemma~\ref{H_basic}~(d) on these $x$ points. As every point not on $l$ is
collinear with at most one of these $x$ points, it follows that
\begin{eqnarray*}
\sum_{P\notin l}(w(P)-1)
&\le& \frac{x\delta q^2}{x-1}\le \delta
(q^2+1)+\frac{2\delta}{q^2-2} < \delta(q^2+1)+1.
\end{eqnarray*}
Hence, $\sum_{P\notin l}(w(P)-1) \le \delta(q^2+1)$.

All but at most two points of $l$ lie on a line of $\B_1$, so $l$ is the only line
of $\L$ on these at least $q^2-1$ points. At most two points of $l$ can be
contained in more than one line of $\L$, but each such point is contained in at
most $\delta+1$ lines of $\L$ by Lemma~\ref{Qminus_basic}~(f). Hence $\sum_{P\in
l}(w(P)-1) \le 2\delta$, and therefore $W\le\delta(q^2+3)$.
\end{proof} 

\begin{Le}\label{bound_on_bq2plus1}
If $\delta\le q-2$, then
\[
\tilde{b}_{q^2+1}\ge q^4+q-\delta-\frac{(q^5+2q^3-2q\delta -q+2)\delta}{q^2-\delta}.
\]
\end{Le}
\begin{proof}
We count the number of incident pairs $(P,l)$, $P \in \M$ and $l$ a
line of $\h(4,q^2)$, to see
\[
|\M|(q^3+1) = |\L|(q^2+1) + \sum_{i=1}^{q^2+1}\tilde{b}_i i \,.
\]
As $\h(4,q^2)$ has $(q^3+1)(q^5+1)=|\L| +
\sum_{i=1}^{q^2+1}\tilde{b}_i$ lines, then
\begin{eqnarray*}
|\L|q^2+\sum_{i=1}^{q^2+1}\tilde{b}_i(i-1)&=&
|\L|(q^2+1)+\sum_{i=1}^{q^2+1}\tilde{b}_ii-(q^3+1)(q^5+1)
\\
 & = & |\M|(q^3+1)-(q^5+1)(q^3+1) \\
 & = & (q^3+1)(q^3+q^2+\delta(q^2+1)) - W(q^3+1) \\
 & \ge & (q^3+1)(q+1)q^2 - 2\delta(q^3+1).
\end{eqnarray*}
From Lemmas \ref{Qminus_basic} (c) and (e) and Lemma~\ref{WboundH}, we have
\[
(q^2-\delta)\sum_{i=2}^{\delta+1}\tilde{b}_i(i-1) \leq
(q^2-\delta)\sum_{i=2}^{\delta+1}b_i(i-1)\le (q^5-q^3)(q^2+1)\delta+2\delta^2.
\]
Together this gives
\[
(|\L|+\tilde{b}_{q^2+1})q^2\ge
(q^3+1)(q+1)q^2-2\delta(q^3+1)-\frac{(q^5-q^3)(q^2+1)\delta+2\delta^2}{q^2-\delta}.
\]
Using $|\L|=q^3+1+\delta$, the assertion follows.
\end{proof}

\begin{Le}\label{le:analyse2}
If $\delta \le q-4$, then $|\L|(|\L|-1)3q <
\tilde{b}_{q^2+1}(q^2+1)q^2$.
\end{Le}
\begin{proof}
First note that by the assumption on $\delta$, we may use the lower bound for
$\tilde b_{q^2+1}$ from the previous lemma, and so we find
\begin{align*}
&
(q^2-\delta)\left(\tilde b_{q^2+1}(q^2+1)q^2-|\L|(|\L|-1)3q\right)
\\
& \ge 
(q-4-\delta)(q^6-\delta)(q^3+q^2+5q+5\delta+21) + r(q,\delta)\,, 
\end{align*}
with 
\begin{align*}
r(q,\delta) & = (81+33\delta+5\delta^2)q^6+(1-2\delta+2\delta^2)q^5+(\delta+7\delta^2)q^4\\
&+(-2\delta^2-6\delta)q^3-\delta q^2+(\delta+3\delta^2+3\delta^3)q-84\delta -41\delta^2-5\delta^3  
\end{align*}
Since $r(q,\delta) \ge 0$ if $\delta \le q-4$, the lemma follows.
\end{proof}

\begin{Le}\label{le:h4qfinal}
If $\L$ contains no pencil, then $\delta \geq q-3$.
\end{Le}
\begin{proof}
Assume that $\delta < q-3$. Consider a hermitian variety $\h(3,q^2)$, denoted
by $\mathcal{H}$, contained in $\h(4,q^2)$. A cover of $\mathcal{H}$ contains at
least $q^3+q$ lines by \cite{M1998}, so $\mathcal{H}$ contains at least one hole
$P$. Of all lines through $P$ in $\h(4,q^2)$, $q^3-q$ are not contained in
$\mathcal{H}$. They must all meet a line of $\L$, so at most $q+1+\delta$ lines
of $\L$ can be contained in $\mathcal{H}$. Hence, at most
$|\L|+(q+1+\delta)q^2=2q^3+q^2+1+\delta(q^2+1)<(q^2+1)(2q+\delta+1)$ points of 
$\mathcal{H}$ are covered.

Now count triples $(l_1,l_2,g)$, where $l_1,l_2$ are skew lines of
$\L$ and $g \not \in \L$ is a line meeting $l_1$ and $l_2$ and being completely
contained in $\M$. Then
\[
|\L|(|\L|-1)z\ge \tilde{b}_{q^2+1}(q^2+1)q^2\,,
\]
where $z$ is the average number of transversals, contained in $\M$ but not
belonging to $\L$, of two skew lines of $\L$. By Lemma~\ref{le:analyse2}, we
find that $z > 3q$. So there exist skew lines $l_1$ and $l_2$ in $\cal L$ such
that at least $z$ transversals to both lines are contained in $\M$. These
transversals are pairwise skew, so the $H(3,q^2)$ induced in the $3$-space
generated by $l_1$ and $l_2$ contains at least $z(q^2+1)\ge
3q(q^2+1)>(q^2+1)(2q+\delta+1)$ points of $\M$. This is a contradiction. 
\end{proof}

We have shown that $\delta \geq q-3$ if $\L$ contains no pencil. Note that we
have no result for $q \in \{2,3\}$. Hence, we have proved the following result.

\begin{Th}\label{pr:h4q2result}
If $\L$ is a generator blocking set of $\h(4,q^2)$, $q>3$, $|\L|=q^3+1+\delta$,
$\delta < q - 3$, then $\L$ contains a pencil of $q^3+1$ lines through a point.
\end{Th}

\section{Polar spaces of higher rank}\label{sec:rankn}

Consider two point sets $V$ and $\B$ in a projective space, $V \cap \B =
\emptyset$. The {\em cone with vertex $V$ and base $\B$}, denoted by $V\B$, is
the set of points that lie on a line connecting a point of $V$ with a point of
$\B$. If $\B$ is empty, then the cone is just the set $V$.

In this section, we denote a polar space of rank $r$ by $\S_{r}$. The parameters
$(s,t)$ refer in this section always to $(q,q)$, $(q,q^2)$, $(q^2,q^3)$
respectively, for the polar spaces $\q(2n,q)$, $\q^-(2n+1,q)$, $\h(2n,q^2)$. The
term {\em polar space} refers from now on always to a finite classical polar
space. Consider a point $P$ in a polar space $\S$. If $\S$ is determined by a
polarity $\phi$ of the ambient projective space, which is true for all polar
spaces except for $\q(2n,q)$, $q$ even, then $P^\perp$ denotes the hyperplane
$P^\phi$. The set $P^\perp \cap \S$ is exactly the set of points of $\S$
collinear with $P$, including $P$. For any point set $A$ of the ambient
projective space, we define $A^\perp := \langle A \rangle^\phi$. 

For $\S=\q(2n,q)$, $q$ even, $P$ a point of $\S$, $P^\perp$ denotes the tangent
hyperplane to $\S$ at $P$. For any point set $A$ containing at least one point
of $\S$, we define the notation $A^\perp$ as
\[
A^\perp := \bigcap_{X \in A \cap \S} X^\perp\,.
\]

Using this notation, we can formulate the following property. Consider
any polar space $\S_n$ of rank $n$, and any subspace $\pi$ of dimension $l \leq
n-1$, completely contained in $\S_n$. Then it holds that $\pi^\perp \cap \S_n =
\pi\S_{n-l-1}$, a cone with vertex $\pi$ and base $\S_{n-l-1}$ a polar space of
the same type of rank $n-l-1$ \cite{Hirschfeld,HT1991}.

A minimal generator blocking set of $\S_{n}$, $n \geq 3$, can be
constructed in a cone as follows. Consider an $(n-3)$-dimensional subspace
completely contained in $\S_n$, hence $\pi_{n-3}^\perp \cap \S_{n} = \pi_{n-3}\S_2$.
Suppose that $\L$ is a minimal generator blocking set of $\S_2$, then $\L$
consists of lines. Each element of $\L$ spans together with $\pi_{n-3}$ a
generator of $\S_n$, and these $|\L|$ generators of $\S_n$ constitute a minimal
generator blocking set of $\S_n$ of size $|\L|$.

Using the smallest generators blocking sets of the mentioned polar spaces of
rank 2, we obtain examples of the same size in general rank, listed in 
Table~\ref{tab:ex}. The notation $\pi_i$ refers to an $i$-dimensional subspace.
When the cone is $\pi_i B$, the example consists of the generators through the
vertex $\pi_i$, contained in the cone $\pi_iB$, meeting the base of the cone in
the elements of the base set, and the size of the example equals the size of the
base set. We will call $\pi_i$ the {\em vertex} of the generator blocking set.

\begin{table}[h!]
\begin{center}
\begin{tabular}{|l|l|l|l|l|}
\hline
polar space & $(s,t)$ & cone & base set & dimension\\
\hline
$\q(2n,q)$ & $(q,q)$ & $\pi_{n-2}\q(2,q)$ & $\q(2,q)$ & $n+1$\\
 &  & $\pi_{n-3}\q^+(3,q)$ & a spread of $\q^+(3,q)$ & $n+1$\\
\hline
$\q^-(2n+1,q)$ & $(q,q^2)$ & $\pi_{n-2}\q^-(3,q)$ & $\q^-(3,q)$ & $n+2$\\
& & $\pi_{n-3}\q(4,q)$ & a cover of $\q(4,q)$ & $n+2$\\
\hline
$\h(2n,q^2)$ & $(q^2,q^3)$ & $\pi_{n-2}\h(2,q^2)$ & $\h(2,q^2)$ & $n+1$\\
\hline
\end{tabular}
\caption{small examples in rank $n$}\label{tab:ex}
\end{center}
\end{table}

The natural question is whether these examples are the smallest ones. The answer is
yes, and the following theorem, proved by induction on $n$, gives slightly more
information.

\begin{Th}\label{th:rankn}
\begin{compactenum}[\rm (a)]
\item Let $\L$ be a generator blocking set of $\q(2n,q)$, with $|\L| = q +
1 + \delta$. Let $\epsilon$ be such that $q+1+\epsilon$ is the size of the
smallest non-trivial blocking set in $\pg(2,q)$. If $\delta <
\mathrm{min}\{\frac{q-1}{2}, \epsilon\}$, then $\L$ contains one of the two
examples listed in Table~\ref{tab:ex} for $\q(2n,q)$. 
\item Let $\L$ be a generator blocking set of $\q^-(2n+1,q)$, with $|\L| = q^2 +
1 + \delta$. If $\delta \le \frac{1}{2}(3q-\sqrt{5q^2+2q+1})$, then $\L$
contains one of the two examples listed in Table~\ref{tab:ex} for $\q^-(2n+1,q)$.
\item Let $\L$ be a generator blocking set of $\h(2n,q^2)$, $q>3$, with $|\L| =
q^3 + 1 + \delta$. If $\delta < q-3$, then $\L$ contains the example listed in
Table~\ref{tab:ex} for $\h(2n,q^2)$.
\end{compactenum}
\end{Th}

\subsection{Preliminaries}

The following technical lemma will be useful.

\begin{Le}\label{le:cases}
\begin{compactenum}[\rm (a)]
\item If a quadric $\pi_{n-4}\q^+(3,q)$ or $\pi_{n-3}\q(2,q)$ in $\pg(n,q)$ is
covered by generators, then for any hyperplane $T$ of $\pg(n,q)$, at least $q-1$
of the generators in the cover are not contained in $T$.
\item If a quadric $\pi_{n-4}\q(4,q)$ or $\pi_{n-3}\q^-(3,q)$ in $\pg(n+1,q)$ is
covered by generators, then for any hyperplane $T$, at least $q^2-q$ of the
generators in the cover are not contained in $T$. 
\item If a hermitian variety $\pi_{n-3}\h(2,q^2)$ in $\pg(n,q^2)$ is covered by
generators, then for any hyperplane $T$ of $\pg(n,q^2)$, at least $q^3-q$ of the
generators in the cover are not contained in $T$.
\end{compactenum}
\end{Le}
\begin{proof}
\begin{compactenum}[\rm (a)]
\item This is clear if $T$ does not contain the vertex of the quadric (i.e. the
subspace $\pi_{n-4}$, $\pi_{n-3}$ respectively). If $T$
contains the vertex, then going to the quotient space of the vertex, it is
sufficient to handle the cases $\q(2,q)$ and $\q^+(3,q)$. The case $\q(2,q)$ is
degenerate but obvious, since any line contains at most two points of $\q(2,q)$.
So suppose that $C$ is a cover of $\q^+(3,q) \subset \pg(3,q)$, then $T$ is a
plane. If $T \cap \q^+(3,q)$ contains lines, then it contains exactly two lines
of $\q^+(3,q)$. Since at least $q+1$ lines are required to cover $\q^+(3,q)$,
at least $q-1$ lines in $C$ do not lie in $T$. 
\item Again, we only have to consider the case that $T$ contains the vertex, and
so it is sufficient to consider the two cases $\q^-(3,q)$ and $\q(4,q)$ in the
quotient geometry of the vertex $T$. For $\q^-(3,q)$, the assertion is obvious. Suppose
finally that $C$ is a cover of $\q(4,q)\subset \pg(4,q)$. Then $T$ has dimension
three. If $T\cap \q(4,q)$ contains lines at all, then $T\cap \q(4,q)$ is a
hyperbolic quadric $\q^+(3,q)$ or a cone over a conic $\q(2,q)$. As these can be
covered by $q+1$ lines and since a cover of $\q(4,q)$ needs at least $q^2+1$
lines, the assertion is obvious also in this case.
\item Now we only have to handle the case $\h(2,q^2)$.
Since all lines of $\pg(2,q^2)$ contain at most $q+1$ points of $\h(2,q^2)$, the
assertion is obvious.
\end{compactenum}
\end{proof}

From now on, we always assume that $\S_{n} \in \{\q(2n,q), \q^-(2n+1,q),
\h(2n,q^2)\}$. In this section, $\L$ denotes a generator blocking set of size
$|\L| = t+1+\delta$ of a polar space $\S_{n}$. 

Section~\ref{sec:rank2} was devoted to the case $n=2$ of
Theorem~\ref{th:rankn}~(b)~and~(c), the case $n=2$ of Theorem~\ref{th:rankn} (a)
is Proposition~\ref{pro:q4q}. The case $n=2$ serves as the induction basis. The
induction hypothesis is that if $\L$ is a generator blocking set of $\S_{n}$ of
size $t+1+\delta$, with $\delta < \delta_0$, then $\L$ contains one
of the examples listed in Table~\ref{tab:ex}. The number $\delta_0$ can be
derived from the case $n=2$ in Theorem~\ref{th:rankn}.

The polar space $\S_{n}$ has $\pg(2n+e,s)$ as ambient projective space. Here $e
= 1$ if and only if $\S_{n} = \q^-(2n+1,q)$, and $e=0$ otherwise. Call a point
$P$ of $\S_{n}$ a {\em hole} if it is not covered by a generator of $\L$. If $P$
is a hole, then $P^\perp$ meets every generator of $\L$ in an
$(n-2)$-dimensional subspace. In the polar space $\S_{n-1}$, which is induced in
the quotient space of $P$ by projecting from $P$, these $(n-2)$-dimensional
subspaces induce a generator blocking set $\L'$, $|\L'| \le |\L|$. Applying the
induction hypothesis, $\L'$ contains one of the examples of $\S_{n-1}$ described
in Table~\ref{tab:ex}, living in dimension $n+e$; we will denote this example by
$\L^P$. Hence, the $(n+1+e)$-space on $P$ containing the $(n-2)$-dimensional
subspaces that are projected from $P$ on the elements of $\L^P$, is a cone with
vertex $P$ and base the $(n+e)$-dimensional subspace containing a
minimal generator blocking set of $\S_{n-1}$ described in Table~\ref{tab:ex}. We
denote this $(n+1+e)$-space on $P$ by $S_P$.

\begin{Le}\label{fourspace}
Consider a polar space $\S_{n} \in \{\q(2n,q),\q^-(2n+1,q),\h(2n,q^2)\}$, and a
generator blocking set of size $t+1+\delta$. If $P$ is a hole and $T$ an
$(n+e)$-dimensional space $\pi$ on $P$ and in $S_P$, then at least $t-\frac{t}{s}$
generators of $\L$ meet $S_P$ in an $(n-2)$-dimensional subspace not contained
in $T$.
\end{Le}
\begin{proof}
This assertion follows by going to the quotient space of $P$, and using
Lemma~\ref{le:cases} and the induction hypothesis of this section.
\end{proof}

We recall the following facts from \cite{HT1991}. Consider a quadric $\Q$ in a
projective space $\pg(n,q)$. An $i$-dimensional subspace $\pi_i$ of $\pg(n,q)$
will intersect $\Q$ again in a possibly degenerate quadric $\Q'$. If $\Q'$ is
degenerate, then $\pi_i \cap \Q = \Q' = R \Q''$, where $R$ is a subspace completely
contained in $\Q$, and where $\Q''$ is a non-singular quadric. We call $R$ the
{\em radical} of $\Q'$. Clearly, all generators of $\Q'$ contain $R$. We recall
that $\Q''$ does not have necessarily the same type as $\Q$.

Consider a hermitian variety $\mathcal{H}$ in a projective space $\pg(n,q^2)$.
An $i$-dimensional subspace $\pi_i$ of $\pg(n,q^2)$ will intersect $\mathcal{H}$
again in a possibly degenerate hermitian variety $\mathcal{H}'$. If
$\mathcal{H}'$ is degenerate, then $\pi_i \cap \mathcal{H} = \mathcal{H}' = R
\mathcal{H}''$, where $R$ is a subspace completely contained in $\mathcal{H}$,
and $\mathcal{H}''$ is a non-singular hermitian variety. We call $R$ the {\em
radical} of $\mathcal{H}'$. Clearly, all generators of $\mathcal{H}'$ contain
$R$.

\begin{Le}\label{fivespace}
Let $\L$ be a minimal generator blocking set of size $t+1+\delta$ of $\S_{n}$.
If an $(n+1+e)$-dimensional subspace $\Pi$ of $\pg(2n+e,s)$ contains more than
$\frac{t}{s}+1+\delta$ generators of $\L$, then $\L$ is one of the examples
listed in Table~\ref{tab:ex}.
\end{Le}
\begin{proof}
First we show that $\Pi$ is covered by the generators of $\L$. Assume not and let
$P$ be a hole of $\Pi$. If $\Pi\cap \S_{n}$ is degenerate, then its radical is
contained in all generators of $\Pi\cap \S_{n}$, so $P$ is not in the radical.
Hence, $P^\perp\cap \Pi$ has dimension $n+e$ and thus $S_P\cap \Pi$ has dimension at
most $n+e$. Lemma~\ref{fourspace} shows that at least $t-\frac{t}{s}$ generators
of $\L$ meet $S_P$ in an $(n-2)$-subspace that is not contained in $\Pi$. Hence,
$\Pi$ contains at most $\frac{t}{s}+1+\delta$ generators of $\L$. This
contradiction shows that $\Pi$ is covered by the generators of $\L$. 

The subspace $\Pi$ is an $(n+1+e)$-dimensional subspace containing generators of
$\S_{n}$. This leaves a restricted number of possibilities for the structure of
$\Pi \cap \S_{n}$: $\Pi \cap \S_{n} \in \{\pi_{n-3}\q^+(3,q),
\pi_{n-2}\q(2,q)\}$ when $\S_n = \q(2n,q)$, $\Pi \cap \S_{n} \in
\{\pi_{n-4}\q^+(5,q), \pi_{n-3}\q(4,q), \pi_{n-2}\q^-(3,q)\}$ when $\S_n =
\q^-(2n+1,q)$, and  $\Pi\cap \S_{n} \in \{\pi_{n-3} \h(3,q^2),
\pi_{n-2}\h(2,q^2)\}$ when $\S_{n} = \h(2n,q^2)$.

{\bf Case 1:} $\Pi\cap \S_{n} = \pi_{n-2} \S_{1}$ ($\S_{1}=\q(2,q), \q^-(3,q)$, or
$\h(2,q^2)$).
\\
A generator of $\L$ contained in $\Pi$ contains the vertex $\pi_{n-2}$. If one of
the $t+1$ generators on $\pi_{n-2}$ is not contained in $\L$, then at least $s$
generators of $\L$ are required to cover its points outside of $\pi_{n-2}$. Hence,
if $x$ of the $t+1$ generators on $\pi_{n-2}$ are not contained in $\L$, then
$|\L|\ge t+1-x+xs$. Since $|\L| = t+1+\delta$, with $\delta < s-1$, this implies
$x=0$. So $\L$ contains the pencil of generators of $\pi_{n-2}\S_{1}$, and by the
minimality of $\L$, it is equal to this pencil.

{\bf Case 2:} $\Pi\cap \S_{n} \in \{\pi_{n-3}\q^+(3,q), \pi_{n-3} \q(4,q)\}$.
\\
Recall that $\Pi \cap \S_{n} = \pi_{n-3}\q^+(3,q)$ when $\S_n = \q(2n,q)$ and then
$(s,t) = (q,q)$, and that $\Pi \cap \S_{n} = \pi_{n-3}\q(4,q)$ when $\S_n =
\q^-(2n+1,q)$ and then $(s,t) = (q,q^2)$.

All generators of $\L$ contained in $\Pi$ must contain the vertex $\pi_{n-3}$.
We will show that the generators of $\L$ contained in $\Pi$ already cover $\Pi\cap
\S_{n}$; then $\L$ contains (by minimality) no further generator and thus
$\L$ is one of the two examples. 

Assume that some point $P$ of $\Pi\cap\S_n$ does not lie on any generator of $\L$
contained in $\Pi$. As all generators of $\L$ contained in $\Pi$ contain the vertex
$\pi_{n-3}$, then $P$ is not in this vertex. Hence, $P^\perp\cap \Pi\cap \S_{n}$ is a
pencil of $\frac{t}{s}+1$ generators $g_0,\dots,g_{\frac{t}{s}}$ on the
subspace $\pi_{n-2}=\erz{P,\pi_{n-3}}$. None of the generators $g_i$ is contained
in $\L$. Therefore, at least $s+1$ generators of $\L$ are required to cover
$g_i$. One such generator of $\L$ may contain the vertex $\pi_{n-2}$ and counts
for each generator $g_i$, but this still leaves at least $(\frac{t}{s}+1)s+1$ generators
in $\L$ necessary to cover all the generators $g_i$. But $|\L|<t+s$, a contradiction.

{\bf Case 3:} $\Pi \cap \S_{n} \in \{\pi_{n-4}\q^+(5,q), \pi_{n-3}\h(3,q^2)\}$, and we
 will show that this case is impossible.\\
Recall that $\Pi \cap \S_{n} = \pi_{n-4}\q^+(5,q)$ when $\S_n = \q^-(2n+1,q)$ and
then $(s,t) = (q,q^2)$, and that $\Pi \cap \S_{n} = \pi_{n-3}\h(3,q^2)$ when
$\S_n = \h(2n,q^2)$ and then $(s,t) = (q^2,q^3)$. In both cases, $\frac{t}{s} =
q$. Denote by $V$ the vertex of $\Pi\cap \S_n$.

All generators of $\L$ contained in $\Pi$ must contain the vertex $V$. We will
show that the generators of $\L$ contained in $\Pi$ already cover $\Pi\cap
\S_{n}$.

Assume that some point $P$ of $\Pi\cap\S_n$ does not lie on any generator of $\L$
contained in $\Pi$. As all generators of $\L$ contained in $\Pi$ contain the vertex $V$, 
then $P$ is not in $V$. When $\S_n = \q^-(2n+1,q)$, then $P^\perp\cap \Pi\cap \S_{n}$
contains $2(q+1)$ generators on the subspace $\pi = \erz{P,V}$. None of these 
generators is contained in $\L$. These $2(q+1)$ generators split into two classes, 
corresponding with the two classes of generators of the hyperbolic quadric $\q^+(3,q)$, 
the base of the cone $\pi \q^+(3,q) = P^\perp\cap \Pi\cap \S_{n}$. Consider one 
such class of generators, denoted by $g_0,\dots,g_{q}$. 
When $\S_n = \h(2n,q^2)$, then $P^\perp\cap \Pi\cap \S_{n}$ contains $q+1$ generators 
on the subspace $\pi = \erz{P,V}$, and none of these generators is contained in $\L$. 
Also denote these generators by $g_0,\dots,g_{q}$. 
So in both cases we consider $\frac{t}{s}+1=q+1$ generators $g_0,\dots,g_{q}$ on 
the subspace $\pi = \erz{P,V}$, not contained in $\L$. Consider now any 
generator $g_i$, then at least $s+1$ generators of $\L$ are required to cover
$g_i$. One such generator of $\L$ may contain the vertex $\pi$ and counts
for each generator $g_i$, but this still leaves at least $(\frac{t}{s}+1)s+1$
generators in $\L$ necessary to cover all the generators $g_i$. But $|\L|<t+s$,
a contradiction. 

Hence in the quotient geometry of the vertex $V$, the generators of $\L$ contained in $\Pi$ 
induce either a cover of $\q^+(5,q)$, which has size at least $q^2+q$ 
(see \cite{ES2001}) or a cover of $\h(3,q^2)$, which has size at least $q^3+q^2$ 
(see \cite{M1998}). In both cases, this is a contradiction with the assumed upper 
bound on $|\L|$.
\end{proof}

\subsection{The polar spaces $\q^-(2n+1,q)$ and $\h(2n,q^2)$}

This subsection is devoted to the proof of Theorem~\ref{th:rankn}~(b)~and~(c).

\begin{Le}\label{conic}
Suppose that $\C$ is a line cover of $\q(4,q)$ with $q^2+1+\delta$ lines. Then each
conic and each line of $\q(4,q)$ meets at most $(\delta+1)(q+1)$ lines of $\C$.
\end{Le}
\begin{proof} If $w(P)+1$ is defined as the number of lines of $\C$ on a point
$P$, then the sum of the weights $w(P)$ over all points of $\q(4,q)$ is
$\delta(q+1)$. Hence, a conic can meet at most $(\delta+1)(q+1)$ lines of $\C$,
and the same holds for lines.
\end{proof}

\begin{Le}\label{le:general1}
Suppose that $\S_{n} \in \{\q^-(2n+1,q),\h(2n,q^2)\}$, $n \geq 3$. Suppose that
$\L$ is a minimal generator blocking set of size $t+1+\delta$ of $\S_{n}$,
$\delta < \delta_0$. If there exists a hole $P$ that projects $\L$ on a
generator blocking set containing a minimal generator blocking set of $\S_{n-1}$
that has a non-trivial vertex, then $\L$ is one of the examples in
Table~\ref{tab:ex}.
\end{Le}
\begin{proof}
Let $P$ be the hole that projects $\L$ on an example with a vertex $\alpha$.
Hence, there exists a line $l$ on $P$ in $S_P$ meeting at least $t+1$ of the
generators of $\L$, and the vertex of $S_P$ equals $\erz{P,\alpha}$. We have
$l^\perp\cap \S_{n}=l \S_{n-2}$, hence the number of planes completely contained
in $\S_{n}$ on the line $l$ equals $|\P_{n-2}|$ ($\P_{n-2}$ is the point set of
$\S_{n-2}$). 

Suppose that a generator $g$ of $\L$ meets such a plane $\pi$ in a line,
then this line intersects $l$ in a point $P'\neq P$. But then $l^\perp \cap g$ has
dimension $n-2$, so the number of lines on $P'$ contained in $l^\perp \cap g$
equals $\theta_{n-3}$, and so $\theta_{n-3}$ planes of $\S_{n}$ on $l$ meet $g$
in a line.

Denote by $\lambda$ the number of planes on $l$ contained completely in the
vertex of $S_P$. Then $\lambda$ equals the number of points in a hyperplane of
$\alpha$; when $\alpha$ is a point, then $\lambda = 0$. Then there are $|\P_{n-2}|
-\lambda$ planes on $l$, completely contained in $\S_n$, but not contained in
the vertex of $S_P$.

Consequently, we find such a plane $\pi$ meeting the vertex of $S_P$ only in
$l$, and meeting at most $m := |\L|\cdot\theta_{n-3}/(|\P_{n-2}|-\lambda)$
generators $g_i$ in a line. A calculation shows that $m < 2$ if $n \geq 3$.
Hence, from the at least $t+1$ generators of $\L$ that meet $l$, at most one
meets $\pi$ in a line, and the at most $\delta$ generators of $\L$ that do not
meet $l$ can meet $\pi$ in at most one point. Hence, $\pi$ contains a hole $Q$
not on $l$.

At least $t+1$ generators of $\L$ meet $S_P$ in an $(n-2)$-dimensional subspace 
and meet the line $l$, and at least $t+1$ generators of $\L$ meet $S_Q$ in an
 $(n-2)$-dimensional subspace. Hence,  at least $2(t+1)-|\L| = t+1-\delta$
generators of $\L$ meet both $S_P$ and $S_Q$ in an $(n-2)$-dimensional 
subspace, and meet the line $l$.

Suppose that the projection of $\L$ from $Q$ contains a generators blocking set 
with a non-trivial vertex $\alpha'$. It is not possible that $l_Q$ is contained in 
$\alpha'$, since then all elements of $\L$ meeting $S_Q$ in an 
$(n-2)$-dimensional subspace would meet $\pi$ in a line, a contradiction to $m < 2$.

The base of $\L^Q$ is either a parabolic quadric $\q(4,q)$, an elliptic quadric 
$\q^-(3,q)$ or a hermitian curve $\h(2,q^2)$. In the latter two cases, since 
neither $\q^-(3,q)$ nor $\h(2,q^2)$ contain lines, the projection of the line $l$  
from $Q$, denoted by $l_Q$, is not contained in the base of $\L^Q$. Suppose 
now that the base of $\L^Q$ is a parabolic quadric $\q(4,q)$, and that this base 
contains the line $l_Q$. The $t+1-\delta$ generators of $\L$ meeting both $S_P$ 
and $S_Q$ in an $(n-2)$-dimensional subspace, all meet $l$. 
These $t+1-\delta$ generators are projected on generators of $\L^Q$, meeting the 
base of $\L^Q$ in a cover. Hence, in the quotient geometry of the vertex of $\L^Q$, 
$l_Q$ is now a line of $\q(4,q)$ meeting at least $t+1-\delta=q^2+1-\delta$ lines of a 
cover of $\q(4,q)$, a contradiction with Lemma~\ref{conic},  since 
$t+1-\delta > (\delta+1)(q+1)$ if $\delta_0 \leq q/2$.

We conclude that the line $l_Q$ is neither contained in the vertex of $\L^Q$ nor 
in the base of $\L^Q$. (This excludes also the possibility that $\L^Q$ has a trivial 
vertex, which is only possible for $n=3$ and $\S_n = \q^-(7,q)$). Hence, $l_Q$
is a line meeting $\alpha'$ and the base of $\L^Q$, and there exists a line $l'\neq l$
in $\pi$ connecting $Q$ with a point of $\alpha'$.

The $t+1-\delta$ generators of $\L$ meeting both $S_P$ and $S_Q$ in an
$(n-2)$-dimensional subspace also meet $l'$ in a point. At most one of these
generators meets $\pi$ in a line, so at least $t-\delta$ of these generators are
projected from the different points $P$ and $Q$ on  generators through a common
point, so before projection, these $t-\delta$ generators of $\L$ must meet in
the common point $X := l \cap l'$.

Now consider a hole $R$ not in the perp of $X$. Then $S_R$ meets at least
$(t-\delta+t+1)-(t+1+\delta) = t-2\delta$ of the generators on $X$ in an
$(n-2)$-subspace. These generators are therefore contained in $T:=\erz{S_R,X}$.
Finally, consider a hole $R'$ not in $T$ and not in the perp of $X$. Then at
least $t-3\delta > \frac{t}{s}+1+\delta$ of the generators that contain $X$
and are contained in  $T$ meet $S_{R'}$ in an $(n-2)$-subspace. These generators
lie therefore in $\erz{S_{R'}\cap T,X}$, which has dimension $n+1+e$. Now
Lemma~\ref{fivespace} completes the proof.
\end{proof}

\begin{Co}\label{co:easierones}
Theorem~\ref{th:rankn} (c) is true for $\h(2n,q^2)$, $n \ge 3$.
\end{Co}
\begin{proof}
Theorem~\ref{pr:h4q2result} guarantees that the assumption of
Lemma~\ref{le:general1} is true for $\S_{n}=\h(2n,q^2)$ and $n=3$.
Theorem~\ref{th:rankn} (c) then follows from the induction hypothesis.
\end{proof}

We may now assume that $\S_{n} = \q^-(2n+1,q)$, $n=3$, and that the projection of
$\L$ from every hole contains a generator blocking set with a trivial vertex,
i.e. a  cover of $Q(4,q)$. As $n=3$, then $\L$ is a set of planes.

\begin{Le}\label{sixspace}
If a hyperplane $T$ contains more than $q+1+3\delta$ elements of $\L$, then
$\L$ is one of the two examples in $\q^-(7,q)$ from Table~\ref{tab:ex}.
\end{Le}
\begin{proof}
Denote by $\L'$ the set of the generators of $\L$ that are contained in $T$.
If $P$ is a hole outside of $T$, then $S_P$ meets all except at most $\delta$
planes of $\L$ in a line, and hence more than $q+1+2\delta$ of these planes are
contained in $T$. Recall that $S_P$ is a cone with vertex $P$ over $S_P\cap T$, and
$S_P\cap T$ has dimension 4. 

Note that $P^\perp \cap \q^-(7,q) = P\Q_5$ with $\Q_5$ an elliptic quadric
$\q^-(5,q)$, and we may suppose that $\Q_5 \subseteq T$. Denote by $\Q_4$ the
parabolic quadric $\q(4,q)$ contained in $\Q_5$ such that $S_P = P\Q_4$, then $T
\cap S_P \cap \q^-(7,q) = \Q_4$. Consider any point $Q \in (\q^-(7,q) \cap
P^\perp) \setminus (S_P \cup \Q_5)$. Clearly $W := Q^\perp \cap T \cap S_P$
meets $\q^-(7,q)$ in an elliptic quadric $\q^-(3,q)$. There are $(q^4-q^2)(q-1)$
points like $Q$, and at most $(q^2-q)(q+1)$ of them are covered by
elements of $\L$, since we assumed that more than $q+1+3\delta$ elements of $\L$
are contained in $T$. So at least $q^5-q^4-2q^3+q^2+q > 0$ points of $(\q^-(7,q)
\cap P^\perp) \setminus (S_P \cup \Q_5)$ are holes and have the property that
$W:=Q^\perp\cap T\cap S_P$ meets $\q^-(7,q)$ in an elliptic quadric $\q^-(3,q)$.
As before, $S_Q\cap T$ has dimension four and meets at least $|\L'|-\delta$
planes of $\L'$ in a line. Then at least $|\L'|-2\delta$ planes of $\L'$ meet
$S_P\cap T$ and $S_Q\cap T$ in a line. As $S_P\cap S_Q\cap T\subseteq W$ does
not contain singular lines, it follows that these $|\L'|-2\delta$ planes of
$\L'$ are contained in the subspace $H:=\erz{S_P\cap T,S_Q\cap T}$.

We have $W\cap \q^-(7,q)=\q^-(3,q)$, so in the quotient geometry of $P$, 
the $|\L'|-2\delta$ planes induce $|\L'|-2\delta$ lines all meeting this $\q^-3,q)$. 
Now $\L$ is projected from $P$ on a cover of a parabolic quadric $\q(4,q)$ with at
most $q^2+1+\delta$ lines. Then $|\L'|-2\delta$ lines of the cover must meet
more than $q+1$ points of this elliptic quadric $\q^-(3,q)$. It follows that
$S_Q\cap T$ contains more than $q+1$ points of the elliptic quadric $\q^-(3,q)$
in $W$ and hence $W\subseteq S_Q$. Then $S_P\cap T$ and $S_Q\cap T$ meet in $W$,
so the subspace $H$ they generate has dimension five. As
$|\L'|-2\delta>q+1+\delta$ planes of $\L$ lie in $H$, Lemma~\ref{fivespace}
completes the proof.
\end{proof}

\begin{Le}\label{le:q7q}
Suppose that $\L$ is a minimal generator blocking set of size $t+1+\delta$ of
$\q^-(7,q)$, $\delta < \delta_0$. If there exists a hole $P$ that projects
$\L$ on a generator blocking set containing a cover of $\q(4,q)$, then $\L$ is
one of the examples in Table~\ref{tab:ex}.
\end{Le}
\begin{proof}
Consider a hole $P$. Then $S_P \cap \q^-(7,q) = P\q(4,q)$. Denote the base of
this cone by $\Q_4$. The assumption of the lemma is that $\L^P$ is a minimal
cover $\C$ of $\Q_4$. Consider a point $X \in \Q_4$ contained in exactly one
line of $\C$. Then $X^\perp \cap \Q_4 = X\q(2,q)$, and each line on $X$ is
covered completely, so $X^\perp \cap \Q_4$ meets at least $q^2+1$ lines of $\C$.

The lines of $\C$ are projections from $P$ of the intersections of elements
of $\L$ with the subspace $S_P$, call $\C'$ this set of intersections that is
projected on $\C$. Thus the line $h=PX$ of $S_P$ on $P$ meets exactly
one line of $\C$ and $h^\perp \cap S_P \cap \q^-(7,q) = h\q(2,q)$ meets at least
$q^2+1$ lines of $\C'$. At most $\delta$ elements of $\L$ are possibly not
intersecting $S_P$ in an element of $\C'$, so we find a hole $Q$ on $h$ with
$Q\neq P$. There are at least $q^2+1$ elements in $\C'$, so at least
$q^2+1-\delta$ elements come from planes $\pi\in\L$ with $\pi\cap Q^\perp\subset
S_Q$. For each such element, its intersection with $h\q(2,q)$ lies in $S_Q$.
Thus either $S_P\cap S_Q=h^\perp\cap S_P$ or $S_P\cap S_Q$ is a $3$-dimensional
subspace of $h^\perp\cap S_P$ that contains a cone $Y\q(2,q)$.

In the second case, the
vertex $Y$ must be the point $Q$ (as $Q\in S_Q$); but then projecting from $Q$
we see a cover of $\q(4,q)$ containing a conic meeting at least $q^2+1-\delta$
of the lines of the cover. In this situation, Lemma~\ref{conic} gives
$q^2+1-\delta\le (\delta+1)(q+1)$, that is $\delta>q-3$, a contradiction.

Hence, $S_P\cap S_Q$ has dimension four, so $T=\erz{S_P,S_Q}$ is a hyperplane. At
least $q^2$ planes of $\L$ meet $S_P$ in a line that is not contained in
$S_P\cap S_Q$. At least $q^2-\delta$ of these also meet $S_Q$ in a line and
hence are contained in $T$. It follows from $\delta < q/2$ that $q^2-\delta >
q+1+3\delta$, and then Lemma~\ref{sixspace} completes the proof. 
\end{proof}

\begin{Co}\label{co:harderone}
Theorem~\ref{th:rankn} (b) is true for $\q^-(2n+1,q)$, $n \ge 3$.
\end{Co}
\begin{proof}
Theorem~\ref{pr:q5qresult} guarantees that for $\S_{n}=\q^-(7,q)$ and $n=3$,
the assumption of either Lemma~\ref{le:general1} or Lemma~\ref{le:q7q} is true.
Hence, Theorem~\ref{th:rankn} (b) follows for $n=3$. But then the assumption
of Lemma~\ref{le:general1} is true for $\S_{n}=\q^-(2n+1,q)$ and $n=4$,
and then Theorem~\ref{th:rankn} (b) follows from the induction
hypothesis.
\end{proof}

\subsection{The polar space $\q(2n,q)$}

This subsection is devoted to the proof of Theorem~\ref{th:rankn}~(a).
Lemma~\ref{le:general1} can also be translated to this case, but only for a bad
upper bound on $\delta$. Therefore we treat the polar space $\q(2n,q)$ separately.
Recall that for $\q(2n,q)$, $\delta_0 = \mathrm{min}\{\frac{q-1}{2}, 
\epsilon\}$, with $\epsilon$ such that $q+1+\epsilon$ is the size of the
smallest non-trivial blocking set of $\pg(2,q)$.

We suppose that $\L$ is a generator blocking set of $\q(2n,q)$, $n \geq 3$, of
size $q+1+\delta$, $\delta < \delta_0$. Recall that $\L^R$ is the minimal
generator blocking set of $\q(2n-2,q)$ contained in the projection of $\L$ from
a hole $R$. So when $n=3$, it is possible that $\L^R$ is a generator blocking
set of $\q(4,q)$ with a trivial vertex.

For the Lemmas~\ref{le:blocking}, \ref{le:proj:regulus},
and~\ref{le:q6qregulus}, the assumption is that $n=3$, and that for any hole
$R$, $\L^R$ has a trivial vertex, i.e. $\L^R$ is a regulus.

So let $R$ be a hole such that $\L^R$ is a regulus. Let $g_{i}$, $i=1,\ldots,$
$q+1+\delta$, be the elements of $\L$ and denote by $l_i$ the intersection of
$R^\bot\cap g_i$. At least $q+1$ of the lines $l_i$ are projected on the lines
of the regulus $\L^R$. We denote the $q+1$ lines of the regulus $\L^R$ by
$\tilde{l}_i$, $i=1,\ldots,q+1$. The opposite lines of the regulus $\L^R$ are
denoted by $\tilde{m}_i$, $i=1,\ldots,q+1$.

\begin{Le}\label{le:blocking}
Suppose that $\tilde{m}_j$ is a line of the opposite regulus and that $B_j$ is the
set of points that are the intersection of the lines $l_i$ with
$\erz{R,\tilde{m}_j}$. Then $B_j$ contains a line.
\end{Le}
\begin{proof}
Since at least $q+1$ lines $l_i$ must meet $\erz{R,\tilde{m}_j}$ in a point,
$|B_j| \geq q+1$. We show that $B_j$ is a blocking set in $\erz{R,\tilde{m}_j}$.
Assume that a line $k$ in $\erz{R,\tilde{m}_j}$ is disjoint to $B_j$ and take a point
$R'$ on $k$, then $R'$ is a hole. By the assumption made before this lemma,
$\L^{R'}$ is also a generator blocking set with a trivial vertex, i.e. a regulus
$\R'$. Consider now the plane $\pi := \langle R, k \rangle$. The plane $\pi$ is
contained in $S_R$. If the plane $\pi$ is also contained in $S_{R'}$, then it is
projected from $R'$ on a line of $\R'$ or of the opposite regulus of $\R'$; in both
cases it is projected on a covered point of $\R'$, and hence the line $k$ must
contain an element of $B_j$, a contradiction. So the plane $\pi$ is not contained
in $S_{R'}$.

There are at least $q+1$ elements of $\L$ that meet $S_{R'}$ in a line; such a
line is projected from $R'$ on a line of $\R'$. No two lines that are projected
on two different lines of $\R'$ can meet $\pi$ in the same point. Hence, of the
at least $q+1$ elements of $\L$ that are projected from $R'$ on $\R'$, at most
one can meet $\pi$ in a point, since otherwise $\pi$ is projected from $R'$ on a
line of the opposite regulus of $\R'$, but then the plane $\pi$ would be
contained in $S_{R'}$. But then at most $\delta+1$ elements of $\L$ can meet
$\pi$ in a point, a contradiction with $|B_j| \geq q+1$.
\end{proof}
We denote the line contained in the set $B_j$ by $m_j$, and so $m_j$ is
projected from $R$ on $\tilde{m}_j$. Now we consider again the hole $R$ and the
regulus $\L^R$.

\begin{Le}\label{le:proj:regulus}
The generator blocking set $\L^R$ arises as the projection from $R$ of a
regulus, of which the lines are contained in the elements of $\L$.
\end{Le}
\begin{proof}
An element $g_i \in \L$ that is projected from $R$ on the line $\tilde{m}_j$
must meet the plane $\langle R,\tilde{m}_j \rangle$ in a line. But an element
$g_i \in \L$ cannot meet a plane $\langle R,\tilde{l}_i\rangle$ and a 
plane $\langle R, \tilde{m}_j \rangle$ in a line, since then $g_i$ would be a
generator of $\q(6,q)$ contained in $R^\perp$ not containing $R$, a
contradiction. So at most $\delta$ elements of $\L$ meet $S_R$ in a line that is
projected on a line $\tilde{m}_j$. Hence, at least $q+1-\delta$ planes
$\erz{R,\tilde{m}_j}$ do not contain a line $l_i$, so, by
Lemma~\ref{le:blocking},
there are at least $q+1-\delta$ lines $m_j \subseteq B_j$ not coming from the
intersection of an element of $\L$ and $S_R$, that are projected on a line of
the opposite regulus of $\L^R$. Number these $n \ge q+1-\delta$ lines from $1$
to $n$.

Suppose that $l_1,l_2,\ldots,l_{q+1}$ are transversal to $m_1$. Since $\delta\le
\frac{q-1}{2}$, a second transversal $m_2$ has at least $\frac{q+3}{2}$ common
transversals with $m_1$. So we find lines $l_1,\ldots,l_{\frac{q+3}{2}}$ lying
in the same 3-space $\erz{m_1,m_2}$. A third line $m_j$, $j \neq 1,2$, has at
least 2 common transversals with $m_1$ and $m_2$, so all transversals $m_j$ lie
in $\erz{m_1,m_2}$. Suppose that we find at most $q$ lines $l_1,\ldots,l_q$ which are
transversal to $m_1,\ldots,m_{q+1-\delta}$. Then $q+1-\delta$ remaining
points on the lines $m_j$ must be covered by the $\delta+1$ remaining lines
$l_i$, so $\delta +1 \geq q+1-\delta$, a contradiction with the assumption on
$\delta$. So we find a regulus of lines $l_1,\ldots,l_{q+1}$ that is projected
on $\L^R$ from $R$.
\end{proof}

\begin{Le}\label{le:q6qregulus}
The set $\L$ contains $q+1$ generators through a point $P$, which are projected
from $P$ on a regulus.
\end{Le}
\begin{proof}
Consider the hole $R$. By Lemma~\ref{le:proj:regulus}, $R^\perp$ contains a 
regulus $\R_1$ of $q+1$ lines $l_i$ contained in planes of $\L$. Denote the
$3$-dimensional space containing $\R_1$ by $\pi_3$. Consider any hole 
$R' \in \q(6,q) \setminus \pi_3^\perp$. By the assumption made before 
Lemma~\ref{le:blocking} and Lemma~\ref{le:proj:regulus}, $R'$ gives rise to a 
regulus $\R_2$ of $q+1$ lines contained in planes of $\L$. Since 
$R' \in \q(6,q) \setminus \pi_3^\perp$, $\R_1 \neq \R_2$. Hence, at least 
$\frac{q+3}{2}$ planes of $\L$ contain a line of both $\R_1$ and $\R_2$ and 
in at most one plane, the reguli $\R_1$ and $\R_2$ can share the same line. 
The reguli $\R_1$ and $\R_2$ define a $4$- or $5$-dimensional space $\Pi$. 

If $\Pi$ is $4$-dimensional, then $\Pi \cap \q(6,q) = \erz{P,\Q}$, for some
point $P$ and some hyperbolic quadric $\q^+(3,q)$, denoted by $\Q$. For $\Q$ we
may choose the hyperbolic quadric containing $\R_1$. There are at least
$\frac{q+1}{2}$ planes of $\q(6,q)$, completely contained in $\Pi$, containing a
line of $\R_1$ and a different line of $\R_2$. These planes are necessarily
planes of $\L$. Consider now a plane $\pi_2$ of $\q(6,q)$, completely contained
in $\Pi$, only containing a line of $\R_1$ and not containing a different line
of $\R_2$. If $\pi_2$ is not a plane of $\L$, it contains a hole $Q$. Then
$Q^\perp$ intersects the at least $\frac{q+1}{2}$ planes of $\L$ on $P$ in a
line, and the projection of these at least $\frac{q+1}{2}$ lines from
$Q$ is one line $l$. If this line $l$ belongs to $\L^Q$, then at least $q$ more
elements of $\L$ are projected from $Q$ on the $q$ other elements of $\L^Q$,
hence, $q+\frac{q+1}{2} \le q+1+\delta$, a contradiction with $\delta <
\frac{q-1}{2}$. Hence, $\pi_2$ is a plane of $\L$, and $\L$ contains $q+1$
generators of $\q(6,q)$ through $P$, which are projected from $P$ on a regulus.

If $\Pi$ is $5$-dimensional, then its intersection with $\q(6,q)$ is a cone
$P\Q$, $\Q$ a parabolic quadric $\q(4,q)$, or a hyperbolic quadric
$\q^+(5,q)$. If $\Pi \cap \q(6,q) = P\q(4,q)$, then the base $\Q$ can be
chosen in such a way that $\R_1 \subset \Q$. But then the same arguments as in
the case that $\Pi$ is $4$-dimensional apply, and the lemma follows. 

So assume
that $\Pi \cap \q(6,q) = \q^+(5,q)$. Consider again the at least $\frac{q+1}{2}$
planes $\pi^1 \ldots \pi^n$, of $\L$ containing a line of $\R_1$ and a different 
line of $\R_2$. Then half of these planes lie in the same equivalence class and
so intersect mutually in a point. We can assume that the two planes $\pi^1$ and
$\pi^2$ intersect in a point $P$, hence, $\langle \pi^1,\pi^2\rangle$ is a
$4$-dimensional space necessarily intersecting $\q(6,q)$ in a cone $P\Q$, $\Q$ a
hyperbolic quadric $\q^+(3,q)$. Clearly, since two different lines of $\R_1$
span $\langle \R_1 \rangle$ (and two different lines of $\R_2$ span $\langle
\R_2 \rangle$), the reguli $\R_1, \R_2 \subseteq \langle \pi^1,\pi^2\rangle$.
But since the planes $\pi^3 \ldots \pi^n$ contain a different line from $\R_1$
and $\R_2$, these at least $\frac{q+1}{2}$ planes of $\L$ are completely contained
in $\langle \pi^1,\pi^2 \rangle$. But then again the same arguments as in the
case that $\Pi$ is $4$-dimensional apply, and the lemma follows.
\end{proof}

From now on we assume that $n \geq 3$, and that there exists a hole $R$ such
that $\L^R$ has a non-trivial vertex $\alpha$. % of dimension $n-3$. 
This means that also for $n=3$, this vertex is non-trivial. This assumption will
be in use for Lemmas~\ref{le:nicepoint}, \ref{le:bigvertex},
\ref{le:special_generator}, \ref{le:anja_argument}, and
Corollary~\ref{co:nice_in_vertex}. Remark that also the induction hypothesis is
used. We will call the $(n-2)$-dimensional subspace $\langle
R,\alpha\rangle$ the {\em vertex of $S_R$}.

A {\em nice point} is a point that lies in at least $q-\delta$ elements of $\L$.
In the next lemma, for $X$ a hole, we denote by $\bar{\L}^X$ the set of
generators of $\L$ that are projected from $X$ on the elements of $\L^X$. Hence,
the generators of $\bar{\L}^X$ intersect $X^\perp$ in $(n-2)$-dimensional
subspaces.

\begin{Le}\label{le:nicepoint}
Call $\alpha$ the vertex of $\L^R$. Then there exists a nice point $N$ on
every line through $R$ meeting $\alpha$. 
\end{Le}
\begin{proof}
Let $l$ be a line on $R$ projecting to a point of $\alpha$, and
consider the planes of $\q(2n,q)$ on $l$. Consider any generator $g \in \L$.
Suppose that $g$ meets two planes $\pi^1$ and $\pi^2$ on $l$ in a line different
from $l$. Then in the quotient geometry of $l$, i.e. $l^\perp \cap \q(2n,q) =
\q(2n-4,q)$, the two planes $\pi^1$ and $\pi^2$ are two points contained in the
generator $l^\perp \cap g$, which is an $(n-3)$-dimensional 
subspace. Hence, any generator $g \in \L$ meets at most $\theta_{n-3}$ planes
through $l$ in a line different from $l$.

If $g$ meets two planes $\pi^1$ and $\pi^2$ on $l$ in only one point not on $l$,
then in the quotient geometry of $l$, the two planes $\pi^1$ and $\pi^2$ are
again two points contained in the generator $l^\perp \cap g$.
Hence, any generator $g \in \L$ meets at most $\theta_{n-3}$ planes 
through $l$ in exactly one point not on $l$. Finally, if a generator $g \in
\L$ meets a plane $\pi^1$ in a line different from $l$ and a plane $\pi^2$ in a
point not on $l$, then $g$ meets also $\pi^2$ in a line different from $l$,
since by the assumption, $g$ also contains a point of $l$.

Hence, for $g\in\L$, $l\not\subseteq g$ implies that $g$ can meet at most
$\theta_{n-3}$ of these planes in one or more points outside of $l$. As $l$ lies
in $\theta_{2n-5}\ge  \theta_{n-3}(q+1)>\frac12|{\cal L}|\theta_{n-3}$ planes of
$\q(2n,q)$, we can choose a plane $\pi$ on $l$ such that at most one generator
of $\cal L$ meets $\pi$ in a line different from $l$ or in exactly one point of
$\pi\setminus l$. Let $Q\in\pi\setminus l$ be on no generator of $\cal L$. Also,
if there is a generator in $\cal L$ meeting $\pi\setminus l$ in a single point $T$,
then choose $Q$ in such a way that this point $T$ does not lie on the line $QR$.

If the generator blocking set ${\cal L}^Q$ in the quotient of $Q$ has a
non-trivial vertex, then $\pi$ is  not a plane of this vertex, since otherwise
all the generators of $\bar{\cal L}^Q$  would meet $\pi$ in a line different from
$l$, but this is a contradiction with the choice of $\pi$. Since $\bar{\cal
L}^Q$ and $\bar{\cal L}^R$ share at least $q+1-\delta$ generators, then
$q+1-\delta$ generators of $\bar{\cal L}^Q$ meet $l$, and at most one of these
contains a point of $\pi\setminus l$. Hence, we find $q-\delta$ generators in
$\bar{\cal L}^Q \cap \bar{\cal L}^R$, each of them meeting $\pi$ in one point,
which is on $l$.

If the generators of $\bar{\cal L}^Q$ are projected from $Q$ on a generator
blocking set with an $(n-3)$-dimensional vertex (and base a conic $\q(2,q)$), then
points in different generators of ${\cal L}^Q$ are collinear only if they are in
the vertex of the cone. But the points of the $q-\delta$ generators on $l$ are
collinear after projection from $Q$. Hence, if two points of these $q-\delta$
generators on $l$ are different, then $l$ is projected from $Q$ on a line of the
vertex of $\L^Q$, so $\pi$ is a plane in the vertex of $S_Q$, a contradiction.
So the $q-\delta$ generators meeting $l$ in a point all meet $l$ in the same
point $X$, and we are done.

Now assume that the generators of $\bar{\cal L}^Q$ are projected from $Q$ on a
generator blocking set with an $(n-4)$-dimensional vertex, and base a regulus
$\R$. Assume that $l$ has no nice point, then at least two of the $q-\delta$
generators do not meet $l$ in a common point. Then $l$ is skew to the vertex of
the cone, since otherwise all the generators of $\bar{\cal L}^Q$  would meet
$\pi$ in a line different from $l$, but this is a contradiction with the choice
of $\pi$. Hence, $l$ is projected from the vertex of $S_Q$ on a line of the
regulus $\R$ or on a line of the opposite regulus $\R'$. But a line of $\R$
meets exactly one line of ${\cal L}^Q$, so $l$ must be projected from the vertex
of $S_Q$ on a line of the opposite regulus $\R'$. This means that each line of 
$\pi$ on $Q$ is met by a generator of $\bar{\cal L}^Q$ in a single point. 
This applies to the line $QR$, so some generator of $\cal L$ meets $\pi$ in a 
point, which lies on the line $QR$. This is a contradiction with the choice of 
$Q$ inside $\pi$.
\end{proof}

\begin{Co}\label{co:nice_in_vertex}
If $R$ is a hole and $N\in R^\perp$ a nice point, then $N$ lies in the vertex of 
$S_R$.
\end{Co}
\begin{proof}
A nice point lies in at least $q-\delta$ generators of $\L$ and at least $q-2\delta \geq
2$ if these must belong to $\L^R$. As two elements of $\L^R$ necessarily meet in
a point of the vertex of $S_R$, the assertion follows.
\end{proof}

\begin{Le}\label{le:bigvertex}
Let $n \ge 4$. If $\beta$ denotes the subspace generated by all nice points,
then $\dim(\beta)\ge n-3$. 
\end{Le}
\begin{proof}
Suppose that $R$ is a hole. If $n \ge 4$, then by the induction hypothesis,
the vertex of $\L^R$ has dimension at least $n-4$. Hence, using
Lemma~\ref{le:nicepoint}, the nice points generate a subspace $\gamma$ of
dimension at least $n-4$. Suppose that $\dim(\gamma)=n-4$, then
$\dim(\gamma^\perp)=n+3 < 2n$, and so we find a hole $P \not \in \gamma^\perp$.
Consider this hole $P$, then the same argument gives us a subspace $\gamma'$
spanned by nice points in $P^\perp$ of dimension at least $n-4$, different from
$\gamma$. So $\dim(\beta)\ge n-3$.
\end{proof}

\begin{Le}\label{le:special_generator}
There exists a hole $R$ and a generator $g$ on the vertex of $S_R$ such that $g$
meets exactly one element of $\L$ in an $(n-2)$-dimensional subspace and such
that all other elements of $\L$ do not meet $g$ or meet $g$ only in points of
the vertex of $S_R$.
\end{Le}
\begin{proof}
First let $n=3$. By the assumption, there exists a hole $R$ such that $\L^R$ has
a non-trivial vertex, which is a point $X$. So the vertex of $S_R$ is the line
$RX$ and has dimension $n-2$.

Now let $n\geq 4$. By Lemma~\ref{le:bigvertex}, we find a subspace $\gamma$ of
dimension $n-3$ spanned by nice points. Consider a hole $R \in \gamma^\perp$.
Clearly, the vertex of $S_R$ will be spanned by the projection of $\gamma$ from
$R$ and $R$, so has dimension $n-2$. 

So for $n\geq 3$, we always find a hole $R$ such that the vertex $V$ of $S_R$ has
dimension $n-2$, and $V = \langle R,\pi_{n-3}\rangle$, $\pi_{n-3}$ the vertex of
$\L^R$. As $\L^R$ consists of the $q+1$ generators of a cone $\alpha \q(2,q)$,
points in different elements of $\L^R$ are collinear only when they are
contained in $\pi_{n-3}$. So the projection from $R$ of any $(n-2)$-dimensional
intersection $\pi_i$ of an element $\L$ and $S_R$, meets at most one element of
$\L^R$ outside of the vertex $\pi_{n-3}$. Hence, before projection, no element
of $\L$ meets two generators of $\q(2n,q)$ on $V$ in points outside of $V$.
Also, at least $q+1$ elements of $\L$ meet $S_R$ in an $(n-2)$-dimensional
subspace that is projected from $R$ on an element of $\L^R$. So at most $\delta$
elements of $\L$ can meet a generator on $V$ in points outside of $V$, and thus
we find a generator of $\q(2n,q)$ on $V$ only meeting elements of $\L$ in points
of $V$.
\end{proof}

\begin{Le}\label{le:anja_argument}
Let $n\geq 3$. There exists an $(n-3)$-dimensional subspace contained in at
least $q$ elements of $\L$.
\end{Le}
\begin{proof}
Consider the special hole $R$ from Lemma~\ref{le:special_generator}. Call
again $V=\erz{R,\pi_{n-3}}$ the vertex of $S_R$, with $\pi_{n-3}$ the
vertex of $\L^R$. Denote the elements of $\L$ intersecting $S_R$ in an
$(n-2)$-dimensional subspace by $g_i$. By Lemma~\ref{le:special_generator}, we
find a generator $g$ on $V$ intersected by a unique element $g_1$ of $\L$ in an
$(n-2)$-dimensional subspace, and intersected by further elements $g_i$ of $\L$
in at most $(n-3)$-dimensional subspaces contained in $V$. So we find a hole $Q
\neq P$, $Q \in g \setminus V$.  

Clearly, at least $q-\delta$ elements of $\L$ that meet $S_R$ in an
$(n-2)$-dimensional subspace, also meet $S_Q$ in an $(n-2)$-dimensional 
subspace and are projected on elements of $\L^Q$. Consider now the hole $Q$, and
suppose that $\L^Q$ is a cone $\pi_{n-4}\R$, $\R$ a regulus. The generator
$g_1$ is projected from $Q$ on a subspace $\tilde{g}_1$ not in $\L^Q$,
since $\tilde{g}_1$ meets at least $q-\delta$ of the projected spaces $g_i$, $i
\neq 1$, in an $(n-3)$-dimensional space, which has larger dimension than the
vertex of $\L^Q$.  But $\tilde{g}_1$ lies in $\pi_{n-4}\R$, since it intersects
at least $q-\delta$ spaces $g_i$ in an $(n-3)$-dimensional subspace. Hence,
$\tilde{g}_1$ meets the $q+1$ elements of $\L^Q$ in different $(n-3)$-spaces and
is completely covered. So the projection of $R$ from $Q$ is covered by elements
of $\L^Q$, and hence, the line $l=\langle R,Q\rangle$ must meet an element of
$\L \setminus\{g_1\}$, a contradiction. So $\L^Q$ is a cone $\pi'_{n-3}\q(2,q)$. 

It follows that $\tilde{g}_1 \in \L^Q$, so $\pi'_{n-3} \subset
\tilde{g}_1$, and $g_1$ and $V$ are projected from $Q$ on $\tilde{g}_1$.
Before projection from $R$, the elements $g_i$ meet $V$ in $(n-3)$-dimensional
subspaces contained in $V$.

The subspace $\pi'_{n-3}$ lies in the projection from $Q$ of elements of $\L$
meeting $\langle \pi'_{n-3},Q\rangle$ in an $(n-3)$-dimensional subspace. But
the choice of $g$ implies that there is only a unique element of $\L$ meeting
$\langle \pi'_{n-3},Q\rangle$ in an $(n-3)$-dimensional subspace and in points
outside of $V$ (the element meeting $g$ in $g_1$), so, at least $q$ other
elements of $\L$ intersect $V$ in the same $(n-3)$-dimensional subspace. 
\end{proof}

The following lemma summarizes in fact Lemmas~\ref{le:nicepoint},
\ref{le:bigvertex}~and~\ref{le:special_generator}, \ref{le:anja_argument}, and
Corollary~\ref{co:nice_in_vertex}. The condition on $\delta$ enables the use of
the induction hypothesis.

\begin{Le}\label{le:general3}
Let $n \geq 3$. Suppose that $\L$ is a minimal generator blocking set of size
$q+1+\delta$ of $\q(2n,q)$, $\delta \leq \delta_0$. If there exists a hole $R$
that projects $\L$ on a generator blocking set containing a minimal generator
blocking set of $\q(2n-2,q)$ that has a non-trivial vertex, then $\L$ is a
generator blocking set of $\q(2n,q)$ listed in Table~\ref{tab:ex}.
\end{Le}
\begin{proof}
By Lemma~\ref{le:anja_argument}, we can find an $(n-3)$-dimensional subspace
$\alpha$ of $\q(2n,q)$ that is contained in at least $q$ elements of $\L$.
Consider now a hole $H \not \in \alpha^\perp$. Then $H^\perp \cap \alpha^\perp$
is an $(n+1)$-dimensional space containing at least $q-\delta$ intersections of
$H^\perp$ with elements of $\L$ on $\alpha$ through the $(n-4)$-dimensional
subspace $H^\perp \cap \alpha$. Since $S_H$ is $(n+1)$-dimensional, these
$q-\delta$ $(n-2)$-dimensional subspaces lie in the $n$-dimensional space $S_H
\cap \alpha^\perp$. Hence, we find in the $(n+1)$-dimensional space $\langle
\alpha,\S_H\cap \alpha^\perp \rangle$ at least $q-\delta > \delta +2$ elements
of $\L$. Lemma~\ref{fivespace} assures that $\L$ is one of the generator
blocking sets of $\q(2n,q)$ listed in Table~\ref{tab:ex}.
\end{proof}

Finally, we can prove Theorem~\ref{th:rankn} (a).

\begin{Le}\label{co:q2nhigh}
Theorem~\ref{th:rankn} (a) is true for $\q(2n,q)$, $n\geq 3$.
\end{Le}
\begin{proof}
Proposition~\ref{pro:q4q} assures that the assumptions of either
Lemma~\ref{le:q6qregulus} or Lemma~\ref{le:general3}, $n=3$ are true. Hence, 
Theorem~\ref{th:rankn} (a) follows for $n=3$. But then the assumption
of Lemma~\ref{le:general3} is true for $\q(2n,q)$ and $n=4$, and then
Theorem~\ref{th:rankn} (a) follows by induction.
\end{proof}

\section{Remarks}

We mentioned already that a maximal partial spread is in fact a special generator
blocking set. The results of Theorem~\ref{th:rankn} imply an improvement of the
lower bound on the size of maximal partial spreads in the polar spaces
$\q^-(2n+1,q)$, $\q(2n,q)$, and $\h(2n,q^2)$ when the rank is at least $3$.
In Table \ref{tab:spreads}, we summarize the known lower bounds on the size of
small maximal partial spreads of polar spaces. The results for $\q^+(2n+1,q)$,
$\w(2n+1,q)$ and $\h(2n+1,q^2)$ are proved in \cite{KMS}.

\begin{table}
\begin{center}
\begin{tabular}{|c|l|}
\hline
Polar space & Lower bound\\
\hline
$\q^-(2n+1,q)$ & $n\geq 3: q^2+\frac{1}{2}(3q-\sqrt{5q^2+2q+1})$\\
\hline
$\q^+(4n+3,q)$ & $n \geq 1$, $q\geq 7: 2q+1$\\
\hline
$\q(2n,q)$ & $n\geq 3: q+1+\delta_0$, with
$\delta_0=\mathrm{min}\{\frac{q-1}{2}, \epsilon\}$, \\
& $\epsilon$ such that $q+1+\epsilon$ is the size of the smallest non-trivial \\
& blocking set in $\pg(2,q)$.\\
\hline
$\w(2n+1,q)$ & $n\geq 2$, $q\geq 5: 2q+1$\\
\hline
$\h(2n,q^2)$ & $n\geq 3: q^3+q-2$\\
\hline
$\h(2n+1,q^2)$ & $q\geq 13$ and $n\geq 2: 2q+3$\\
\hline

\end{tabular}
\caption{Bounds on the size of small maximal partial spreads}\label{tab:spreads}
\end{center}
\end{table}

One can wonder what happens with generator blocking sets of the polar spaces
$\q^+(2n+1,q)$, $\w(2n+1,q)$, $q$ odd, and $\h(2n+1,q^2)$. Unfortunately, the
approach presented in Section~\ref{sec:rank2} for these polar spaces, fails,
which makes the completely approach of this paper not usable for these polar
spaces in higher rank.

In \cite{BSS:2010}, an overview of the size of the smallest non-trivial blocking
sets of $\pg(2,q)$ is given. When $q$ is a prime, then $\epsilon =
\frac{q+1}{2}$. So when $q$ is a prime, the condition on $\delta$ in the case of
generator blocking sets of $\q(2n,q)$, $n \ge 3$, drops to $\delta <
\frac{q-1}{2}$.

\section*{Acknowledgements}
The research of the first author was also supported by a postdoctoral research
contract on the research project {\em Incidence Geometry} of the Special Research
Fund ({\em Bijzonder Onderzoeksfonds}) of Ghent University. The research of the
second author was supported by a research grant of the Research Council of Ghent
University. The first and second author thank the Research Foundation Flanders
(Belgium) (FWO) for a travel grant and thank Klaus Metsch for his hospitality
during their stay at the Mathematisches Institut of the Universit\"at Gie\ss{}en.


\begin{thebibliography}{1}

\bibitem{BSS:2010}
A.~Blokhuis, P.~Sziklai, and T.~Sz\H{o}nyi, 
\newblock {\em Blocking sets in projective spaces}, in Current research topics
in Galois geometry, Nova Sci. Publ., New York, ch.~3, pp.~59--82, to appear.

\bibitem{Bruen1970}
A.~A. Bruen.
\newblock Baer subplanes and blocking sets.
\newblock {\em Bull. Amer. Math. Soc.}, 76:342--344, 1970.

\bibitem{DBM2005}
J.~De~Beule and K.~Metsch.
\newblock The smallest point sets that meet all generators of {$H(2n,q^2)$}.
\newblock {\em Discrete Math.}, 294(1-2):75--81, 2005.

\bibitem{ES2001}
J.~Eisfeld, L.~Storme, and P.~Sziklai.
\newblock Minimal covers of the {K}lein quadric.
\newblock {\em J. Combin. Theory Ser. A}, 95(1):145--157, 2001.

\bibitem{Hirschfeld}
J.~W.~P. Hirschfeld.
\newblock {\em Projective geometries over finite fields}.
\newblock Oxford Mathematical Monographs. The Clarendon Press Oxford University
  Press, New York, second edition, 1998.

\bibitem{HT1991}
J.~W.~P. Hirschfeld and J.~A. Thas.
\newblock {\em General {G}alois geometries}.
\newblock Oxford Mathematical Monographs. The Clarendon Press Oxford University
  Press, New York, 1991.
\newblock Oxford Science Publications.

\bibitem{KMS}
A.~Klein, K.~Metsch, and L.~Storme.
\newblock Small maximal partial spreads in classical finite polar spaces.
\newblock {\em Adv. Geom.}, 10:379--402, 2010.

\bibitem{M1998}
K.~Metsch.
\newblock The sets closest to ovoids in {$Q\sp -(2n+1,q)$}.
\newblock {\em Bull. Belg. Math. Soc. Simon Stevin}, 5(2-3):389--392, 1998.
\newblock Finite geometry and combinatorics (Deinze, 1997).

\bibitem{PT2009}
S.~E. Payne and J.~A. Thas.
\newblock {\em Finite generalized quadrangles}.
\newblock EMS Series of Lectures in Mathematics. European Mathematical Society
  (EMS), Z\"urich, second edition, 2009.


\end{thebibliography}
\end{document}